\documentclass[english,letter,11pt,twosided]{article}
\usepackage{amsfonts}
\usepackage[left=2.3cm,right=2.0cm, bottom = 2.3cm, top=2.4cm]{geometry}
\usepackage{babel}
\usepackage{slantsc}
\usepackage{array}
\usepackage{amsmath}
\usepackage{amsthm, mathtools}
\usepackage{amssymb}
\usepackage{enumerate}
\usepackage{mathtools}
\usepackage{bbm}
\usepackage{tikz}
\usepackage{subfigure}
\usepackage{pgfplots}
\usepackage{utopia}
\usepackage[labelsep=period]{caption}
\usepackage{hyperref}
\usepackage{setspace}

\usetikzlibrary{patterns,snakes}

\newtheorem{thm}{Theorem}[section]
\newtheorem{cor}{Corollary}[section]
\newtheorem{lem}{Lemma}[section]
\newtheorem{prop}{Proposition}[section]

\newtheorem{rem}{Remark} [section]

\newtheorem{defn}{Definition}[section]

\newcommand{\ZZ}{{\mathbb Z}}
\newcommand{\RR}{{\mathbb R}}
\newcommand{\CC}{{\mathbb C}}

\newcommand{\GG}{{\mathbb G}}
\newcommand{\SQ}{{\mathcal S}_{q, w}(\GG)
}

\definecolor{PennBlue}{RGB}{001,031,091}
\definecolor{PennRed}{RGB}{153,0,0}
\definecolor{NewBlue}{RGB}{001,031,110}
\definecolor{NewRed}{RGB}{200,0,0}
\hypersetup{
pdfborder = {0 0 0},
    colorlinks,
    citecolor=PennRed,
    filecolor=PennRed,
    linkcolor=PennRed,
    urlcolor=PennRed
}
\begin{document}

\title{\fontsize{21}{21}%
\selectfont {Sparsity and Spatial Localization Measures for \\  Spatially Distributed Systems}\thanks{{\footnotesize The work of N. Motee was supported by the National Science Foundation under award NSF-ECCS-1202517, by the Air Force Office of Scientific Research under award AFOSR-YIP FA9550-13-1-0158, and by the Office of Naval Research under award ONR N00014-13-1-0636. The work of Q. Sun was supported by the National Science Foundation under award NSF-DMS-1109063.}}}
\author{\fontsize{13}{13} {Nader Motee}\thanks{{\footnotesize The author is with the Department of Mechanical Engineering and Mechanics, Packard Laboratory, Lehigh University, Bethlehem, PA. Email address:  {\tt\small motee@lehigh.edu}.}}  \hspace{0.5in}{Qiyu Sun}\thanks{{\footnotesize The author is with the Department of Mathematics, University of Central Florida, Orlando
FL. Email address: {\tt\small qiyu.sun@ucf.edu.}}}  \\
~}
\date{\fontsize{11}{11}\selectfont This version: October 2014 \\
First version: Febrauray 2014} 
\maketitle

\begin{abstract}
\fontsize{10}{10}\selectfont \baselineskip0.5cm
We consider the class of spatially decaying systems, where the underlying dynamics are spatially decaying and the sensing and controls are spatially distributed. This class of systems arise in various applications where there is a notion of spatial distance with respect to which couplings between the subsystems can be quantified using a class of coupling weight functions. We exploit spatial decay property of the underlying dynamics of this class of systems to introduce a class of sparsity and spatial localization measures. We develop a new methodology based on concepts of $q$-Banach algebras of spatially decaying matrices that enables us to establish a relationship between spatial decay properties of spatially decaying systems and their sparsity and spatial localization features. Moreover, it is shown  that the inverse-closedness property of matrix algebras plays a central role in exploiting various structural properties of spatially decaying systems. We characterize conditions for exponentially stability of spatially decaying system over $q$-Banach algebras and prove that the unique solutions of the Lyapunov and Riccati equations over a proper $q$-Banach algebra also belong to the same $q$-Banach algebra. It is shown that the quadratically optimal state feedback controllers for spatially decaying systems are sparse and spatially localized in the sense that they have near-optimal sparse information structures.\newline
\textit{IEEE Transaction on Automatic Control. Under Review.}  \newline
\newline
\end{abstract}
\fontsize{12}{12}\selectfont


\thispagestyle{empty} \newpage \fontsize{11}{11}\selectfont\baselineskip %
0.60cm

\onehalfspacing

\allowdisplaybreaks

\section{Introduction}

In a number of important applications, centralized implementation of automatic control is practically infeasible due to lack of access to centralized information. This necessitates design of dynamical networks with sparse and spatially localized information structures, which is currently one of the outstanding open problems in control systems. A precise and comprehensive analysis of sparsity and spatial localization for large-scale dynamical network does not appear to exist in the context of distributed control systems. In this paper we investigate this problem for a broad class of spatially distributed systems, so-called {\it spatially decaying systems}, which are linear systems with off-diagonally decaying state-space matrices. Examples of such systems include linearized and/or spatially discretized models of spatially distributed power networks with sparse interconnection topologies, multi-agent systems with nearest-neighbor coupling structures \cite{Smale-flock-2007}, arrays of micro-mirrors \cite{Neilson01},  micro-cantilevers \cite{Napoli99}, and sensor networks. These systems belong to the class of spatio-temporal systems, where all relevant signals are indexed by a spatial coordinate in addition to time \cite{bamiehPD02}.

For the class of spatially decaying systems, we show that automatic feedback control mechanisms can be spatially localized using far less sensor measurements and actuators than traditional control design techniques. In particular, we exploit spatial structure of the underlying dynamics of networks and reveal that the quadratically-optimal state feedback controllers for a broad range of real-world dynamical networks are inherently sparse and spatially localized, in the sense that, they have near-optimal sparse information structures.

An important class of spatio-temporal systems includes the class of spatially invariant
systems. This class of systems can be defined over continuous or  (infinite or finite dimensional) discrete spatial domains. Their state space matrices consist of translation-invariant operators such as partial differential operators with constant coefficients, spatial shift operators, spatial convolution operators, general pseudodifferential operator and integral operators \cite{trevesbook,grochenigbook}, or a linear combination of such operators  \cite{HormanderVol1}.  A subclass of spatially invariant systems  is considered in \cite{bamiehPD02, curtain-sasane-2011} where the symmetric spatial invariance property of this class of systems are exploited and techniques from spatial Fourier transforms are applied to study optimal control of linear spatially invariant systems. It is shown that the original optimal control problems in spatial domain can be transformed into a family of finite-dimensional parameterized problems in Fourier domain and treated using standard tools for finite-dimensional linear systems in control theory. Moreover, the seminal work \cite{bamiehPD02} shows that the corresponding quadratically-optimal controllers are spatially invariant and have an inherently semi-decentralized information structures.

In this paper, we consider spatially distributed systems over infinite-dimensional {\it discrete} spatial domains, where the dynamics of individual subsystems are {\it heterogeneous}, and the spatial structure does not necessarily enjoy any particular spatial symmetries. Therefore, standard tools such as Fourier analysis cannot be used to analyze this class of systems. The existing traditional methods to study this class of problems are usually based on notions of Banach algebras. One of earliest works in this area is \cite{bunce85} where the algebraic properties of Riccati equations is studied over $C^{*}$--subalgebras of the space of bounded linear operators on some Hilbert space. More recent effort is reported \cite{curtain-2011} where the algebraic properties of Riccati equations is investigated over noncommutative involutive Banach algebras, which can be considered as a generalization of earlier results of \cite{bunce85}. Our paper is close in spirit to earlier works  \cite{mjieee09} and \cite{mjieee08} where a general approach is proposed based on Banach algebras of spatially decaying matrices to analyze the spatial structure of infinite and finite horizon optimal controllers for spatially distributed systems  with arbitrary spatial structures. In \cite{mjieee09} and \cite{mjieee08}, it is shown that quadratically optimal controllers inherit spatial decay properties of the underlying dynamics of the systems. A basic fundamental property of a Banach algebra is that it is a Banach space and locally convex. These nice properties allow us to apply existing methods in the literature to study the class of spatially distributed systems over Banach algebras; for example see \cite{bunce85,Bensoussan-Mitter93,bamiehPD02, mjieee09, mjieee08, motee-sun-CDC-2013, bottcher2011algebraic, curtain-2011, curtain-correction-2013}.

The main contribution of this paper is the development of a unified methodology to determine the degrees of sparsity and spatial localization for a broad class of spatially distributed systems. We categorize the largest classes of spatially distributed systems for which their corresponding quadratically optimal controllers inherent spatial decay property from the dynamics of their underlying systems. We introduce new classes of spatially decaying systems that are defined over $q$-Banach algebras endowed with matrix $q$-norms, where $q$ is an exponent strictly greater than $0$ and less than or equal to $1$. The class of spatially distributed systems considered in \cite{mjieee08, mjieee09} and \cite{curtain-2011}  are special examples of our new class of systems which correspond to exponent $q=1$. When the exponent $q$ is strictly less than $1$, a $q$-Banach algebra becomes locally nonconvex and is not a Banach space. However, they exhibit an interesting property: if $q_2$ is less than $q_1$, then the space of linear systems over $q_1$-Banach algebra is a subset of the space of linear systems defined over $q_2$-Banach algebra. This property implies that when exponent $q$ tends towards $0$, the space of spatially decaying systems over the $q$-Banach algebra  starts to expand and cover larger classes of spatially distributed systems.

The analysis and synthesis of optimal controllers for linear systems involves using inverse operation of matrices or transfer matrices.  In general, linear system analysis over the space of sparse matrices is hopeless, in the sense that sparsity is not preserved under inverse operation. For example, a Toeplitz band matrix belongs to the space of sparse matrices while its inverse may not live in that space. Therefore, the space of sparse matrices is not inverse-closed.  We establish a connection between a notion of sparsity and the spatial decay property of the class of spatially decaying systems. The bridge connecting these two fundamental notions is built upon the key idea of asymptotically approximating the space of sparse matrices by inverse-closed $q$-Banach algebras for sufficiently small values of $q$. We begin in section \ref{sec-II} by categorizing the class of admissible coupling weight functions in order to model coupling structures in spatially distributed systems. In section \ref{sec-III}, we introduce the class of spatially decaying systems over the Gr\"ochenig-Schur class of spatially decaying matrices, which are examples of $q$-Banach algebras. The problem formulation is discussed in section \ref{sec:LQR}, where we define the LQR problem for the class of spatially decaying systems and assert that there is an inherent relationship between sparsity and spatial decay property of the LQR feedback controller. It is argued in section \ref{sec-III} that it is necessary to study spatially distributed systems over $q$-Banach algebras for the range of exponents $0$ to $1$ in order to exploit their sparsity features. Section  \ref{sec-IV} characterizes various algebraic properties of $q$-Banach algebras and shows under some conditions $q$-Banach algebras are inverse-closed. Furthermore, we show that proper $q$-Banach algebras enjoy spectral-invariance property with respect to the space of bounded linear operators on $\ell^2(\GG)$. In section \ref{sec:exp-stability}, we show that characterization of exponential stability for linear systems over $q$-Banach algebras slightly differs  from the standard characterization and  explicitly quantify the decay rate of the $q$-norm of the $C_0$-semigroup. The main control-theoretic results of this paper are in sections \ref{sec-V} and \ref{sec:riccati} where we prove that the unique solution of the Lyapunov equation for linear autonomous systems and Riccati equation resulting from the LQR problems over a proper $q$-Banach algebra also belong to the same $q$-Banach algebra. The significance of these results are discussed in section \ref{sec-VI} and it is shown that  the underlying information structure of the LQR controller for a spatially decaying system is inherently spatially localized and each local controller needs to receive state information only from some neighboring subsystems. Then, we characterize a fundamental limit that explains to what degree a stabilizing controller with spatially decaying structure can be sparsified and spatially localized. Moreover, we argue that a fundamental tradeoff emerges between a desired degree of sparsification and localization and the global performance loss. A probabilistic method and a computational algorithm is proposed in section \ref{sec-near-ideal-meas}  to analyze and compute near-optimal degrees of sparsity and spatial localization for the Gr\"ochenig-Schur class of spatially decaying matrices. We end in section \ref{sec-dis-concl} with a discussion of related areas in which our methodology can be applied, as well as a discussion of some open research problems.

\noindent {\bf Mathematical Notation.}
Throughout the paper, the underlying discrete spatial domain of a spatially distributed system is denoted by $\GG$ which is a subset of $\ZZ^{d}$ for $d \geq 1$.
For a given discrete spatial domain $\GG$, the $\ell^{0}$--measure of  a  vector $x=[x_i]_{i\in \GG}$  is defined as
\begin{equation}
\|x\|_{\ell^{0}(\GG)}~:=~\mathbf{card}\big\{x_{i}\neq 0~\big|~i\in \GG \big\},\label{ell-0-norm}
\end{equation}
where $\mathbf{card}$ is the number of elements in a set.  The value of the $\ell^{0}$--measure represents the total number of nonzero entries in a vector. The $\ell^q$--measure of $x$ is defined by
\[\|x\|_{\ell^{q}(\GG)}^{q}=\sum_{i\in \GG} |x_i|^q\]
for all $0 < q < \infty$, and \[\|x\|_{\ell^{\infty}(\GG)}=\sup_{i \in \GG} |x_{i}|,\]
for $q=\infty$. Whenever it is not ambiguous, we use the simplified notation $\|x\|_q$ for the $\ell^q$--measure of vector $x$.

\begin{rem}
The proof of all theorems, propositions, and lemmas in Sections \ref{sec-II} through \ref{sec-IV} are given in Appendix \ref{appendix}.
\end{rem}

\section{Admissible Coupling Weight Functions}\label{sec-II}
The structure of couplings between subsystems in a   spatially distributed linear systems can be modeled using coupling weight functions. We consider the class of infinite-dimensional linear dynamical networks whose coupling structures are spatially decaying. For this class of systems, the coupling strength between subsystems decays by distance in the spatial domain $\GG$. We say that a function $\rho: \GG \times \GG \rightarrow \RR$  is a {\it quasi-distance function} on $\GG$ if it satisfies the following properties:
\begin{description}
\item[(i)] $\rho(i, j)\ge 0$~ for all ~$i,j\in \GG$;
\item[(ii)] $\rho(i, j)=0$~ if and only if ~$i=j$; and
\item[(iii)] $\rho(i, j)=\rho(j, i)$~ for all ~$i,j\in \GG$.
\end{description}

A quasi-distance function is different from a distance function in that the triangle inequality is not required to hold. In the case that $\GG$ corresponds to an unweighted undirected graph of a sparsely connected spatially distributed system, one may use the shortest distance on a graph to define the quasi-distance $\rho(i,j)$ from vertex $i$ to vertex $j$. A  {\it coupling weight function} is a positive function $w$ that is defined on $\GG \times \GG$ and satisfies the following properties:
\begin{description}
\item[(i)] $w(i, j)\ge 1$ for all $i, j\in \GG$;
\item[(ii)] $w(i,j)=w(j,i)$ for all $i,j\in \GG$; and
\item[(iii)] $\sup_{i\in \GG} w(i, i)<\infty$.
\end{description}

In our framework, we are interested in coupling weight functions that are {\it submultiplicative}, i.e.,
\begin{equation}\label{weight.submultiplication}
 w(i, j) \le  w(i,k) \hspace{0.05cm} w(k, j)
\end{equation}
for all $i,j,k \in \GG$. Later on in our analysis, this property will help us to define matrix norms that enjoy submultiplicative property. In order to develop a relevant matrix space to study the sparsity and spatial localization in spatially distributed systems, we need to impose some additional technical conditions on coupling weight functions that will guarantee certain growth rates for the coupling weight functions.

\begin{defn}\label{def-weight}
Suppose that $\rho:\GG \times \GG\longmapsto \RR$
is a quasi-distance function. A coupling weight function $w=[w(i,j)]_{i,j\in \GG}$ is called {\it admissible} whenever there exist a companion weight function  $u=[u(i,j)]_{i,j\in \GG}$, an exponent $\theta \in (0,1)$, and a positive constant $D$ such that
\begin{equation}\label{uw.eq}
w(i,j)\le w(i,k)\hspace{0.05cm} u(k,j) + u(i,k)\hspace{0.05cm} w(k,j)  \quad {\rm  for \ all} \ i,j,k\in \GG, \end{equation}
and the following inequality
\begin{eqnarray}\label{theta.condition}
 & & \hspace{-1.35cm} \sup_{i\in \GG} \bigg\{ \inf_{\tau\ge 0}
\bigg[\sum_{j\in \GG \atop \rho(i,j)< \tau} |u(i,j)|^{\frac{2q}{2-q}}\bigg]^{1-\frac{q}{2}} ~+~ t~ \sup_{j\in \GG \atop\rho(i,j)\ge \tau} \Big( \frac{u(i,j)}{w(i,j)}\Big)^{q}\bigg\} ~\le~ D t^{1-\theta}
\end{eqnarray}
holds for all $t\ge 1$ when $0 < q \leq 1$, and
\begin{eqnarray}\label{theta.condition2*}
 & & \hspace{-1.3cm} \sup_{i\in \GG} \bigg\{\inf_{\tau\ge 0}
\bigg[\hspace{0cm}\sum_{j\in \GG \atop \rho(i,j)< \tau} |u(i,j)|^2\bigg]^{\frac{1}{2}} ~+~  t ~\bigg[\sum_{j\in \GG \atop\rho(i,j)\ge \tau} \left( \frac{u(i,j)}{w(i,j)}\right)^{\frac{q}{q-1}}\bigg]^{\frac{q-1}{q}}\bigg\} ~\le~ D t^{1-\theta}
\end{eqnarray}
holds for all $t\ge 1$ when $1<q\le \infty$.
\end{defn}

We refer the reader to \cite[P. 3102]{suntams07} for similar definitions on the admissibility of a weight function for $q\ge 1$.  The following lemma characterizes two classes of most popular coupling weight functions that appear in modeling of various real-world dynamical networks. We refer to Appendix \ref{weight.appendix} for the proof,  c.f. Example 2.2 in \cite{suntams07}.


\begin{figure}[t]
\begin{center}
\includegraphics[trim = 20mm 0mm 0mm 10mm, clip, width=0.5\textwidth]{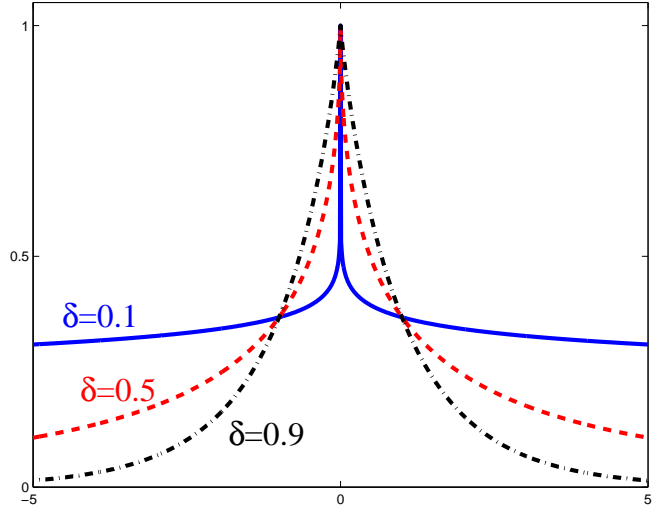}
\caption{{\small The decay rates of the inverse of the sub-exponential coupling weight function \eqref{sub-exp-w} with $d=1$ and $\sigma=1$ are depicted for three different values of parameter $\delta$. }}
\label{fig-subexp-w}
\end{center}
\vspace{-.8cm}
\end{figure}

\begin{lem}\label{weightexample.lem} Suppose that the quasi-distance function on $\ZZ^d$ is defined by $\rho(i,j)=\|i-j\|_{\infty}$. For $0<q\le 1$, the class of sub-exponential coupling weight functions
\begin{equation}
e_{\sigma, \delta}:=\big[e_{\sigma, \delta}(i,j)\big]_{i,j\in \ZZ^d}=\bigg[e^{\left(\frac{\|i-j\|_\infty}{\sigma}\right)^\delta}\bigg]_{i,j\in \ZZ^d} \label{sub-exp-w}
\end{equation}
with parameters $\sigma>0$ and $\delta\in (0, 1)$, are submultiplicative  and admissible with constants
\[ D_{e}= 2^{d+1-\frac{dq}{2}}, \quad \theta_{e}=\frac{q\delta (2-2^\delta)}{q\delta+d(1-\frac{q}{2})\sigma^{\delta}},\]
and the class of polynomial coupling weight functions
\begin{equation}
\pi_{\alpha, \sigma}~:=~\big[\pi_{\alpha, \sigma} (i,j)\big]_{i,j\in \mathbb{Z}^d}~=~\left[\left(\frac{1+\|i-j\|_\infty}{\sigma}\right)^{\alpha}\right]_{i,j\in \mathbb{Z}^d}\label{poly-w}
\end{equation}
with parameters $\alpha,\sigma>0$, are submultiplicative and admissible with constants
\[D_{\pi}=2^{q\alpha+1} \max \Big\{ 1, (2\sigma)^{d(1-\frac{q}{2})}\Big\}, \quad \theta_{\pi}= \frac{q\alpha}{q\alpha+ d(1-\frac{q}{2})}.\]
\end{lem}

In real-world applications, the decay rate of the coupling coefficients in state space matrices determines the type of a suitable coupling weight function for that specific dynamical network (see Figures \ref{sub-exp-w} and \ref{poly-w} for some examples).

\begin{rem}
A weak version of the submultiplicative property \eqref{weight.submultiplication} can deduced from Definition \ref{def-weight} as follows:
\[
 w(i, j) ~\le~ C_0  w(i,k) \hspace{0.05cm} w(k, j)
 \quad {\rm  for \ all} ~~\ i,j,k\in \GG.
\]
with constant $C_{0}\le 2 \sqrt[q]{D}$. We refer to Appendix \ref{append-remark} for more details. We should also emphasize that the class of weight functions that are considered in this section are more general than the class of weight functions introduced earlier in \cite{gltams06, suntams07, sunca11}.
\end{rem}

\section{The Class of Spatially Decaying Systems}\label{sec-III}
We consider the class of infinite-dimensional linear systems whose dynamics are governed by
\begin{eqnarray}
\dot{x}& = &Ax+Bu, \label{linear-sys-1}\\
y & = & C x + D u, \label{linear-sys-2}
\end{eqnarray}
where it is assumed that all state-space matrices are constant by time and all relevant signals of the system are indexed by a spatial coordinate in addition to time.  The state, input and output variables are represented by infinite-dimensional vectors $x=[x_{i}]_{i \in \GG}, u=[u_{i}]_{i \in \GG}, y=[y_{i}]_{i \in \GG}$, respectively. The linear system \eqref{linear-sys-1}-\eqref{linear-sys-2} is defined on the state space $\ell^2(\GG)$. We are interested in a class of infinite-dimensional linear systems \eqref{linear-sys-1}-\eqref{linear-sys-2} with the common property that there is a notion of spatial distance with respect to which couplings between the subsystems can be quantified using a class of coupling weight functions. This class of systems are so called spatially decaying systems and are defined over the following general class of spatially decaying matrices.

\begin{figure}[t]
\begin{center}
\includegraphics[trim = 20mm 0mm 0mm 10mm, clip, width=0.5\textwidth]{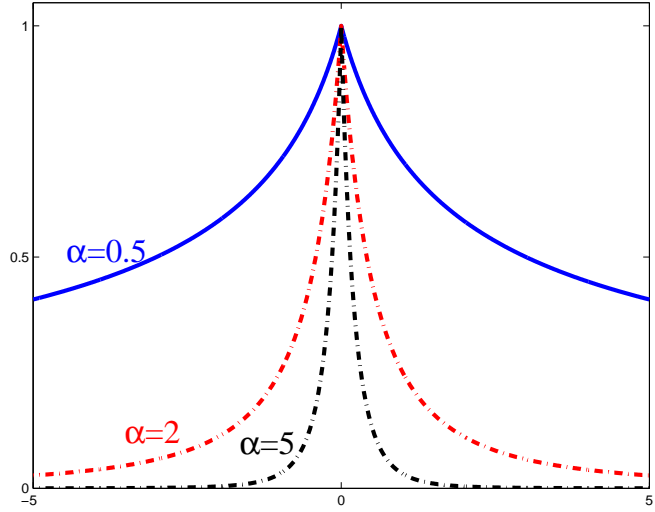}
\caption{{\small The decay rates of the inverse of the polynomial coupling weight function \eqref{poly-w} with $d=1$ and $\sigma=1$ are depicted for three different values of parameter $\alpha$. }}
\label{fig-poly-w}
\end{center}
\vspace{-.8cm}
\end{figure}

\begin{defn}
For a given admissible coupling weight function $w$ on $\GG \times \GG$, the  Gr\"ochenig-Schur class of infinite-dimensional matrices  on $\GG$
is denoted by ${\mathcal S}_{q, w}(\GG)$  and defined as
\begin{equation}\label{sqw.def}
{\mathcal S}_{q, w}(\GG) = \Big\{ A=[a_{ij}]_{i,j\in \GG} ~\big|~
\  \|A\|_{{\mathcal S}_{q, w}(\GG)}<\infty \Big\}
\end{equation}
for $0<q\le \infty$, where the $\mathcal{S}_{q, w}$--measure is defined by
\begin{eqnarray}
& & \hspace{-1.3cm} \|A\|_{{\mathcal S}_{q, w}(\GG)}:=
\max\Big\{\sup_{i\in \GG} \Big(\sum_{j\in \GG} |a_{ij}|^q w(i,j)^{q}\Big)^{1/q}, ~ \sup_{j\in \GG} \Big(\sum_{i\in \GG} |a_{ij}|^q w(i,j)^{q}\Big)^{1/q}\Big\},\label{sqwnorm.def}
\end{eqnarray}
for all $0<q<\infty$, and
\begin{equation}\label{sinfinitywnorm.def}
 \|A\|_{{\mathcal S}_{\infty, w}(\GG)}:=
\sup_{i, j\in \GG} |a_{ij}| w(i,j),
\end{equation}
for $q=\infty$.
\end{defn}
\vspace{0.1cm}

We should remark that this class of spatially decaying matrices was studied  for $q = 1$ in \cite{gltams06, mjieee08, mjieee09}, for $q = \infty$ in \cite{jaffard90}, for $1 \leq q \leq \infty$ in \cite{sunca11,suntams07}, and the results of this paper extends this class of matrices to include range of exponents $0 < q \le 1$.  When $0< q < 1$, the quantity \eqref{sqwnorm.def} is so called a $q$-norm and it is nonconvex as $\ell^{q}$-measures are nonconvex  for this range of exponents. Therefore,  the space of matrices ${\mathcal S}_{q, w}(\GG)$ is not  locally convex and that is not a Banach space. In \cite{motee-sun-CDC-2013}, we verified the subalgebra properties of this class of matrices $0< q \leq 1$.  Whenever it is not ambiguous, for simplicity of our notations we will employ ${\mathcal S}_{q, w}$  instead of ${\mathcal S}_{q, w}(\GG)$.

\begin{defn}
For a fixed  exponent $0 < q \leq \infty$, a linear system \eqref{linear-sys-1}-\eqref{linear-sys-2} is called spatially decaying on $\GG$ with respect to an admissible coupling weight function $w$ if  $A, B, C, D \in \mathcal{S}_{q, w}(\GG)$.
\end{defn}

These are systems with off-diagonally decaying state-space matrices. This definition greatly generalizes the earlier works \cite{mjieee09} and \cite{mjieee08}  that only studied the class of such systems for exponent $q=1$. For every exponent $0 < q < 1$, the corresponding class of spatially decaying systems are defined over a matrix space  $\mathcal{S}_{q, w}(\GG)$ that is not a Banach space, which prevents us from employing existing techniques in the literature that are based on continuous integrals on Banach spaces. Therefore, all existing proof methods to study algebraic properties of linear systems cannot be directly applied for systems that are defined over $\mathcal{S}_{q, w}(\GG)$ for $0 < q < 1$. Our results in this paper significantly generalize existing works in the literature to study structural properties of spatially decaying systems.

\begin{rem}
One of the interesting facts about the class of spatially decaying linear systems is that they are closed under basic system operations such as series, parallel, and feedback interconnections. This can be shown by applying the results of Theorems \ref{schur-cor} and \ref{wiener.thm}.
\end{rem}

\begin{rem}
We should emphasize that the result of this paper in the following sections hold for all exponents $0 < q \leq \infty$. Nevertheless, we will only focus on the class of spatially decaying matrices with exponents $0 < q \leq 1$. Our method of analysis and all results can be modified accordingly to cover the range of exponents $1 \le q \leq \infty$.
\end{rem}

\section{Problem Formulation}\label{sec:LQR}
We formulate the problem of  investigating the sparsity and spatial localization features of an infinite horizon Linear Quadratic Regulator (LQR) problem for spatially decaying systems. The LQR problem is defined as the problem of minimizing the quadratic cost functional
\begin{equation}
J=\int_{0}^{\infty} \big(x(t)^{*}Qx(t)+u(t)^{*}Ru(t)\big)dt \label{LQ-cost}
\end{equation}
subject to the dynamics \eqref{linear-sys-1}-\eqref{linear-sys-2} with initial condition $x(0)=x_{0} \in \ell^{2}(\GG)$. Our basic assumption is that the linear system is spatially decaying and that all state-space matrices $A,B,C,D$ as well as weight matrices $Q, R$ belong to $\mathcal{S}_{q, w}(\GG)$. There is a rich literature that consider this problem on Hilbert spaces (see \cite{curtain-book,Bensoussan-Mitter93} and references in there) and show that under some standard assumptions the unique solution to this problem is achieved by a linear state feedback control law $u = -Kx$ for which
\begin{equation}
K=-R^{-1}B^{*}X,
\end{equation}
where $X$ is the unique solution of the following Riccati equation
\begin{equation}
A^{*}X+XA+Q-XBR^{-1}B^{*}X=0.
\end{equation}

In spatially decaying systems, the internal states of the underlying system as well as control inputs are distributed in the spatial domain. We assume that the dynamics of individual subsystems in the spatially decaying system are heterogeneous and the spatial structure does not necessarily enjoy any particular spatial symmetries. Therefore, standard tools such as Fourier analysis cannot be applied to analyze this class of systems. In \cite{mjieee08}, the authors proposed an operator theoretic approach based on the Banach algebras of spatially decaying matrices $\mathcal{S}_{1, w}(\GG)$ to analyze the spatial structure of infinite horizon optimal controllers for spatially distributed systems  with arbitrary spatial structure and showed that quadratically optimal controllers inherit spatial decay properties of the underlying systems. The results of \cite{mjieee08} and \cite{curtain-2011} state that the LQR feedback control law $K=\big[k_{ij}\big]_{i,j \in \GG}$~ is spatially localized, i.e.,
\begin{equation}
|k_{ij}|~\leq~C_{0}~w(i,j)^{-1} \label{decay-LQR-feedback}
\end{equation}
where $C_{0}=\| K \|_{\mathcal{S}_{1, w}(\GG)}$ is a finite number. This implies that the underlying information structure of the optimal control law is sparse and spatially localized on the spatial domain and each local controller needs  to receive state information only from some neighboring subsystems rather than from the entire network.

{\it The Problem:} The goal of this paper is to determine the degrees of sparsity and spatial localization, i.e.,  the communication requirements for the controller array, for the optimal solution of the LQR problem that is defined over a $\mathcal{S}_{q,w}(\GG)$ for $0 < q \leq 1$.

Our primary focus is on revealing
the fundamental role of underlying structure of spatially decaying system in showing that the corresponding LQR state feedback control law has an inherently sparse and localized architecture in the spatial domain. This property can enable us to discover fundamental tradeoffs between sparsity of the information
structure in controller array and global performance loss in a spatially decaying system by exploiting the spatial structure of the underlying system. In the next section, it is discussed that in order to define a viable measure to quantify the sparsity and spatial localization in spatially decaying systems one needs to study the LQR problem over the  Gr\"ochenig-Schur class of matrices $\mathcal{S}_{q,w}(\GG)$ for $0 < q \leq 1$.

\section{The  Space of Sparse Matrices and Their  Asymptotic Approximations}\label{sec-III}

We establish a connection between sparsity and spatial decay features of spatially decaying systems based on the fact that the endowed $\mathcal{S}_{q,w}$--measures are indeed asymptotic approximations of an ideal sparsity measure as exponent $q$ tends to zero. This implies that the Gr\"ochenig-Schur class of spatially decaying matrices can be viewed as asymptotic relaxations of the space of sparse matrices. The range of exponents $0 < q < 1$ is extremely important for our development as we can asymptotically quantify  the sparsity and spatial localization properties of spatially decaying systems for sufficiently small values of exponent $q$. This property comes at the expense of working in matrix spaces that are not Banach spaces.

For every infinite-dimensional vector $x=[x_i]_{i\in \GG}$ with bounded entries, the $\ell^q$--measure approximates the $\ell^{0}$--measure defined by \eqref{ell-0-norm} asymptotically, i.e.,
\begin{equation}\label{sequence.asym}
\lim_{q\to 0}~\|x\|_{\ell^q(\GG)}^q~=~\|x\|_{\ell^0(\GG)}.\end{equation}

This observation motivates us to consider asymptotic behavior of the $\mathcal{S}_{q, w}$--measure as $q$ tends to zero in order to quantitatively identify near--sparse information structures with semi-decentralized architectures for the optimal solution of the LQR problem for spatially decaying systems. For this purpose, let us define an ideal sparsity measure, so called $\mathcal{S}_{0,1}$--measure, for a matrix $A=[a_{ij}]_{i,j \in \GG}$ by
\begin{equation}
\|A\|_{\mathcal{S}_{0,1}(\GG)}:= \max \Big\{\sup_{i \in \GG} \|a_{i \cdot}\|_{\ell^{0}(\GG)} ,\sup_{j \in \GG} \|a_{\cdot j}\|_{\ell^{0}(\GG)} \Big\}, \end{equation}
where $a_{i\cdot}$ is the $i$'th row and $a_{\cdot j}$ is the $j$'th column of matrix $A$. The value of $\mathcal{S}_{0,1}$--measure reflects the maximum number of nonzero entries in all rows and columns of matrix $A$.

\begin{thm}\label{qtozero.thm}
For a given matrix $A=[a_{ij}]_{i,j \in \GG}$ with finite $\mathcal{S}_{0,1}$--measure and bounded entries, i.e., $\|A\|_{{\mathcal S}_{\infty, w}}<\infty$, we have
\begin{equation}\label{lqmatrixlimit}\lim_{q\to 0} ~\|A\|_{{\mathcal S}_{q, w}(\GG)}^q~= ~\|A\|_{{\mathcal S}_{0, 1}(\GG)}.\end{equation}
\end{thm}

The most important implication of this theorem is that small values of $q$ (closer to zero) can lead to reasonable approximations of the ideal sparsity measure for matrices. This is particularly true for  spatially decaying matrices with slowly decaying rates, such as polynomially decaying matrices. However, for matrices with rapidly decaying rates, such as sub-exponentially decaying matrices, larger values of $q$ (closer to one) can also result in reasonable measures for sparsity. We discuss this in Section \ref{sec-VI}.

The proposed $\mathcal{S}_{0,1}$--measure has an interesting interoperation when $A$ is adjacency matrix of an unweighted undirected graph. In this case, the value of $\mathcal{S}_{0,1}$--measure is equal to the maximum node degree in that graph. The sparsity $\mathcal{S}_{0,1}$--measure has several advantages over the conventional sparsity measure
\begin{equation}
\hspace{-0.00cm} \|A\|_{0}=\sum_{i,j \in \GG}|a_{ij}|^{0}= \mathbf{card}\big\{a_{ij}\neq 0 ~\big|~ i,j \in \GG \big\}. \label{sparsity-2}
\end{equation}
The value of the $\mathcal{S}_{0,1}$--sparsity measure reveals some valuable information about sparsity as well as spatial locality features of a given sparse matrix, while \eqref{sparsity-2} does not enjoy this property. Moreover, \eqref{sparsity-2} does not exhibit any interesting algebraic property and may not be useful in infinite-dimensional settings. Let us consider the set of all sparse matrices characterized by
\begin{equation}
{\mathcal S}_{0,1}(\GG)=\Big\{A~ \Big|~\|A\|_{{\mathcal S}_{0,1}(\GG)} < \infty   \Big\}.
\end{equation}

\begin{prop}
The $\mathcal{S}_{0,1}$--measure satisfies the following properties:
\begin{itemize}
\item[{(i)}] $\|\alpha A\|_{\mathcal{S}_{0,1}(\GG)} = \|A\|_{\mathcal{S}_{0,1}(\GG)}$;
\item [{(ii)}]
$\|A+B\|_{\mathcal{S}_{0,1}(\GG)}\le \|A\|_{\mathcal{S}_{0,1}(\GG)}+  \|B\|_{\mathcal{S}_{0,1}(\GG)}$; and
\item [{(iii)}]  $\|AB\|_{\mathcal{S}_{0,1}(\GG)}\le \|A\|_{\mathcal{S}_{0,1}(\GG)} \|B\|_{\mathcal{S}_{0,1}(\GG)}$
\end{itemize}
for all nonzero scalars $\alpha$ and matrices $A, B \in \mathcal{S}_{0,1}(\GG)$.
\end{prop}
\vspace{0.1cm}

Properties (i)-(iii) imply that the set ${\mathcal S}_{0,1}(\GG)$ is closed under addition and multiplication. One of the fundamental properties of $\SQ$ that is not inherited by ${\mathcal S}_{0,1}(\GG)$ is the inverse-closedness property. For instance, a Toeplitz band matrix belongs to ${\mathcal S}_{0,1}(\GG)$ while its inverse may not live in ${\mathcal S}_{0,1}(\GG)$. In fact, ${\mathcal S}_{0,1}(\GG)$ is not even an algebra. A formal definition of inverse-closedness is given in Definition \ref{Def:inv-closed}. We show that inverse-closedness property of matrix algebras plays a central role in exploiting various structural properties of spatially distributed systems 
over such matrix algebras.

\begin{rem}
For large values of $q$ (closer to $\infty$), ${\mathcal S}_{q,w}$--measure asymptotically approximates ${\mathcal S}_{\infty,w}$--measure and its value measures degree of spatial localization. On the other end of the spectrum, ${\mathcal S}_{q,w}$--measure asymptotically approximates ${\mathcal S}_{0,1}$--measure for small values of $q$ (closer to $0$). This observation implies that in order to study sparsity features of spatially distributed systems it is necessary to consider range of exponents $ q \leq 1$. A fundamental challenges emerges here as ${\mathcal S}_{q,w}$ is not a Banach space for $0 < q < 1$.  For mathematical purposes, $q=1$ is the best possible choice as ${\mathcal S}_{1,w}$--measure is indeed a norm and convex. More precisely, the matrix space ${\mathcal S}_{1,w}(\GG)$ is the ``minimal" Banach algebra (and therefore, a Banach space) with respect to exponent $q$ among all matrix spaces ${\mathcal S}_{q,w}(\GG)$ for $1 \leq q \leq \infty$. Furthermore, spatially decaying matrices in ${\mathcal S}_{1,w}(\GG)$ enjoy the fastest decay rates among all families of Gr\"ochenig-Schur class of matrices for $0 < q \leq  1$. In Section \ref{sec-VI}, we propose a rigorous criterion and an algorithm to compute an exponent $0 < q < 1$ in order to  ${\mathcal S}_{q,w}$--measure to become a viable sparsity measure.
\end{rem}

\begin{figure}[t]
\begin{center}
\includegraphics[trim = 0mm 0mm 0mm 0mm, clip, width=0.7\textwidth]{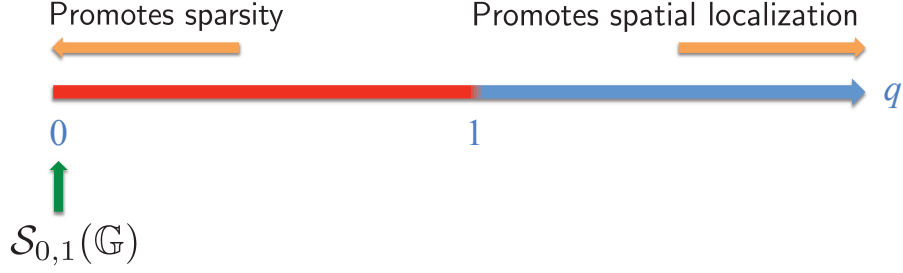}
\caption{{\small For $1 \leq q \leq \infty$, ${\mathcal S}_{q,w}(\GG)$ is a Banach algebra. However, ${\mathcal S}_{q,w}(\GG)$ for $0 < q <1$  is not a Banach space.  The analysis of spatially distributed systems over ${\mathcal S}_{q,w}(\GG)$ for $0 < q \leq 1$ requires new and unconventional techniques based on $q$-Banach algebras.}}
\label{fig-spars-local}
\end{center}
\vspace{-.8cm}
\end{figure}

\section{Sparsity Through $q$-Banach Algebras}\label{sec-IV}

The basic properties of the  Gr\"ochenig-Schur class of spatially decaying matrices pave the way to extend our exploration to more abstract ground in order to capture fundamental properties of spatially decaying systems. In this section, we develop a rigorous mathematical foundation based on the abstract notion of $q$-Banach algebras for range of exponents $0 < q \leq 1$  that enables us to study the sparsity and spatial localization for a general class of spatially decaying systems. The Gr\"ochenig-Schur class of spatially decaying matrices for $0 < q \leq 1$ is an example of a $q$-Banach algebra.

\begin{defn}\label{def-q-banach0}
For  $0<q \leq 1$,  a complex vector space of matrices ${\mathcal A}$ is a {\em $q$-Banach space} equipped with $q$-norm  $\|\cdot\|_{\mathcal A}$  if
it is complete with respect to the metric
$$d_{\mathcal A}(A,B):=\|A-B\|_{\mathcal A}^q$$
for $A, B\in {\mathcal A}$ and the $q$-norm satisfies
\vspace{0.05cm}
 \begin{itemize}
 \item[{(i)}] $\|A\|_{\mathcal A} \ge 0$, and $\|A\|_{\mathcal A}=0$ if and only if $A=0$;
 \item[{(ii)}]  $\|\alpha A\|_{\mathcal A}=  |\alpha|\  \|A\|_{\mathcal A}$; and
 \item[{(iii)}]  $\|A+B\|_{\mathcal A}^q\le \|A\|_{\mathcal A}^q+\|B\|_{\mathcal A}^q$
\end{itemize}
for all $A, B\in {\mathcal A}$ and all $\alpha\in \mathbb{C}$.
\end{defn}

A $q$-Banach space ${\mathcal A}$ with $q=1$ is a Banach space, while it is a quasi-Banach space for any $0<q\le 1$, because
\begin{equation*}
\|A+B\|_{\mathcal A}\le (\|A\|_{\mathcal A}^q+\|B\|_{\mathcal A}^q)^{1/q}
\le 2^{1/q} (\|A\|_{\mathcal A}+\|B\|_{\mathcal A})
\end{equation*}
for all $A, B\in {\mathcal A}$. It is straightforward to verify that  the series $\sum_{n=1}^\infty u_n$ converges
 in the $q$-Banach space  ${\mathcal A}$ if
$\sum_{n=1}^\infty \|u_n\|_{\mathcal A}^q<\infty$.

\begin{defn}\label{def-q-banach}
For  $0<q \leq 1$, a $q$-Banach space ${\mathcal A}$ equipped with $q$-norm  $\|\cdot\|_{\mathcal A}$
 is a {\it $q$-Banach algebra} if
 it contains a unit element $I$, i.e.,
$M:=\|I\|_{\mathcal{A}} < \infty$,  and there exists a constant $K_0 > 0$ such that
\begin{itemize}
\item[{(iv)}] $\|AB\|_{\mathcal A}\le K_{0}  \|A\|_{\mathcal A} \|B\|_{\mathcal A}$ for all $A, B\in {\mathcal A}$.
    \end{itemize}
\end{defn}

One can show that the submultiplicative property for a $q$-Banach algebra holds  by rescaling the  $q$-norm as $\|\cdot\|_{\mathcal A}^*=K_0\|\cdot\|_{\mathcal A}$ and get
$$\|AB\|_{\mathcal A}^*\le  \|A\|_{\mathcal A}^* \|B\|_{\mathcal A}^*, \quad {\rm for\ all} \ \ A, B\in {\mathcal A}.$$

\begin{defn}\label{def-diff}
For a given exponent $0< q \leq 1$ and a Banach algebra ${\mathcal B}$, its $q$-Banach subalgebra ${\mathcal A}$ is a {\it differential} Banach subalgebra of order $\theta \in (0, 1]$ if there exist a constant $D > 0$ such that its $q$-norm satisfies the differential norm property
\begin{equation}
\hspace{0cm} \|AB\|_{\mathcal A}^q ~\le~ D~\|A\|_{\mathcal A}^q~\|B\|_{\mathcal A}^q\left( \left( \frac{\|A\|_{\mathcal{B}}}{\|A\|_{\mathcal{A}}}\right)^{q\theta} + \left(\frac{\|B\|_{\mathcal{B}}}{\|B\|_{\mathcal{A}}}\right)^{q\theta}\right) \nonumber
\end{equation}
for all $A, B\in {\mathcal A}$.
\end{defn}

For a Banach subalgebra ${\mathcal A}$ (i.e., when $q=1$), the above differential norm property
with $\theta\in (0, 1]$
 has been widely used in operator theory and noncommutative
geometry \cite{blackadar91, kissin94, rieffel10}, and  in solving (non)linear functional equations \cite{sunsiam07, sunaicmsubmitted}. The differential norm property plays a crucial role in establishing inverse-closedness property for $q$-Banach algebra, identifying  exponential stability conditions, and exploiting algebraic properties of the unique solutions of Lyapunov and algebraic Riccati equations over $q$-Banach algebras.
One of the most interesting and practical examples of a $q$-Banach algebra is ${\mathcal S}_{q, w}(\GG)$ for $0 < q \leq 1$. It turns out that the $q$-Banach algebra ${\mathcal S}_{q, w}(\GG)$ enjoys the above differential norm property.

\begin{defn}\label{def-proper}
For $0<q \leq 1$,  a
 $q$-Banach algebra ${\mathcal A}$ of  matrices on $\GG$ equipped with $q$-norm $\|\cdot\|_{\mathcal A}$
is said to be  {\it proper}  if it is a $q$-Banach algebra with the following additional properties:

\begin{description}
 \item[{{\bf (P1)}}] $\mathcal{A}$ is closed under the complex conjugate operation, i.e., for all $A\in {\mathcal A}$ we have
\begin{equation}\label{wiener.thm.eq2}
\|A^*\|_{\mathcal A}=\|A\|_{\mathcal A};
\end{equation}
 \item[{{\bf (P2)}}] ${\mathcal A}$ is a subalgebra of ${\mathcal B}(\ell^2(\GG))$ and continuously embedded with respect to it, i.e., for all $A\in {\mathcal A}$ we have
\begin{equation}\label{wiener.thm.eq4}
\|A\|_{{\mathcal B}(\ell^2(\GG))}\le \|A\|_{\mathcal A};
\end{equation}
\item[{{\bf (P3)}}] $\mathcal{A}$ is a differential Banach subalgebra of ${\mathcal B}(\ell^2(\GG))$ of order $\theta\in (0,1]$.
\end{description}
\end{defn}

We should mention that under the assumption \eqref{wiener.thm.eq4},  a differential $q$-Banach subalgebra ${\mathcal A}\subset {\mathcal B}(\ell^2(\GG))$ of order $\theta$ is also  a differential $q$-Banach subalgebra ${\mathcal A}$ of order $\theta'\in (0, \theta)$.

 \begin{thm}\label{schur-cor}
For every $0 < q \leq 1$ and admissible coupling weight function $w$, the space of matrices ${\mathcal S}_{q, w}(\GG)$ is a proper $q$-Banach subalgebra of $\mathcal{B}(\ell^{2}(\GG))$.
\end{thm}

The proof of the above theorem is given in Appendix \ref{schur.appendix}.
For $0<q<1$, the space of matrices $\SQ$ is not a Banach space, since
 the $q$-norm does not satisfy the triangle inequality.

\begin{defn}\label{Def:inv-closed}
A subalgebra ${\mathcal A}$ of a Banach algebra ${\mathcal B}$ is {\it inverse-closed} if $A\in {\mathcal A}$ and   $A^{-1}\in {\mathcal B}$ implies that $A^{-1}\in {\mathcal A}$.
\end{defn}

As we discussed earlier in Section \ref{sec-III},
the space of sparse matrices $\mathcal{S}_{0,1}(\GG)$ is not inverse-closed in general. There are several families of Banach subalgebras of  (infinite-dimensional) matrices with some certain off-diagonal decay properties for which  the inverse-closedness property, also known as Wiener's lemma, is inherited from the parent Banach algebra. We refer the reader to  \cite{balan,
baskakov90,
grochenigklotz10,
grochenigl03,
gltams06,  
shincjfa09, 
suncasp05,
suntams07,  sunca11, sunaicmsubmitted}
and   the  survey papers \cite{grochenigsurvey, Krishtal11}, \cite{shin-sun-2013} for more details. It is known that under certain assumptions on the weight function $w$ the Gr\"ochenig-Schur class of matrices ${\mathcal S}_{q, w}(\GG)$ for $1\le q\le \infty$   is an inverse-closed Banach subalgebra  of  ${\mathcal B}(\ell^2(\GG))$. The following new result shows that under some mild assumptions $q$-Banach subalgebras of ${\mathcal B}(\ell^2(\GG))$ for $0 < q \leq 1$ inherit inverse-closedness property from ${\mathcal B}(\ell^2(\GG))$.

\begin{thm}\label{wiener.thm}
For $ 0 < q \leq 1$, suppose that ${\mathcal A}$ is a proper $q$-Banach subalgebra of ${\mathcal B}(\ell^2(\GG))$. Then ${\mathcal A}$ is inverse-closed in ${\mathcal B}(\ell^2(\GG))$.
\end{thm}

Our proof of Theorem \ref{wiener.thm} in Appendix \ref{wienerthm.proofappendix} is constructive and provides explicit estimates for the $q$-norm of the inverse matrix $A^{-1}$. The results of Theorems \ref{schur-cor} and \ref{wiener.thm} imply that ${\mathcal S}_{q, w}(\GG)$ for $0 < q \leq 1$ is inverse-closed in ${\mathcal B}(\ell^2(\GG))$. The last necessary piece to add to our toolbox is to establish a relationship between the spectral set of a matrix with respect to $ {\mathcal B}(\ell^2(\GG))$
 and its proper $q$-Banach subalgebras.

\begin{thm}\label{thm-spectral}
For $ 0 < q \leq 1$, suppose that ${\mathcal A}$ is a proper $q$-Banach subalgebra of ${\mathcal B}(\ell^2(\GG))$. Then it follows that
\begin{equation}
\sigma_{\mathcal{A}}(A) =\sigma_{{\mathcal B}(\ell^2(\GG))}(A)\quad {\rm for \ all} ~~ A\in {\mathcal A}, \label{spectral-thm}
\end{equation}
where the complement of the spectral set of $A\in \mathcal{A}$ is given by
\begin{equation}
\CC\backslash \sigma_{\mathcal{A}}(A):=\big\{ \lambda \in \CC~\big|~\lambda I - A ~\textrm{has bounded inverse in}~ \mathcal{A}\big\}. \nonumber
\end{equation}
\end{thm}

This spectral-invariance property will help us in the following sections to develop a general framework to study the sparsity and spatial localization features of spatially distributed systems.

\section{Exponential Stability over $q$-Banach Algebras} \label{sec:exp-stability}

For the LQR problem in Section \ref{sec:LQR}, one of the main requirements is that the optimal solution of the LQR problem has to be exponentially stabilizing.  Therefore, we need to revisit the notion of exponential stability for linear systems over $q$-Banach algebras. It is said that a matrix $A$  in a ($q$-)Banach algebra ${\mathcal A}$ is {\em exponentially stable} if there exist strictly positive constants $E$ and $\alpha$ such that
\begin{equation}\label{es.def}
\|e^{tA}\|_{\mathcal A}~\le~ E \hspace{0.05cm}e^{-\alpha t}~~~\textrm{for all}~~~ t\ge 0.
\end{equation}
Then exponential stability of matrix $A$ in a proper $q$-Banach subalgebra $\mathcal{A}$ of ${\mathcal B}(\ell^2(\GG))$ is equivalent to its exponential stability in ${\mathcal B}(\ell^2(\GG))$. The necessity follows from
the fact that ${\mathcal A}$ is continuously imbedded in  ${\mathcal B}(\ell^2(\GG))$.
The sufficiency  follows from the following theorem, which  expresses this relationship in more explicit form by characterizing decay rate of the $q$-norm of the strongly continuous semigroup of $A$.

\begin{thm}\label{exponentialstablesqw.thm}
Suppose that $ 0 < q \leq 1$ and ${\mathcal A}$ is a proper $q$-Banach subalgebra of ${\mathcal B}(\ell^2(\GG))$.
If $A\in {\mathcal A}$ is exponentially stable in ${\mathcal B}(\ell^2(\GG))$, i.e., there exist  some  constants $E,\alpha>0$ such that
\begin{equation}\label{es.def0}
\|e^{tA}\|_{{\mathcal B}(\ell^2(\GG))}\le E e^{-\alpha t}
\end{equation}
for all $t\ge 0$, then
\begin{equation}\label{exponentialst ablesqw.lem.eq1}
\|e^{tA}\|_{\mathcal A}^q ~\le~
C(t)~ e^{-\alpha q t}
\end{equation}
for all $t\ge 0$, where constants
$M, D, K_{0}, \theta$ are defined in Definitions \ref{def-q-banach}, \ref{def-diff} and \ref{def-proper},
\begin{eqnarray}
C(t)& = & M^q  \hspace{0.03cm} \Xi_{0} \hspace{0.03cm} e^{qE}
 \left(2DE^{q\theta}
 \right)^{\omega(t)}, \nonumber \\
\Xi_{0}&=& \sum_{n=0}^\infty ~\frac{ K_{0}^{nq}}{(n!)^q} ~\left(  \frac{\|A\|_{\mathcal A}}{\|A\|_{{\mathcal B}(\ell^2(\GG))}}\right)^{nq},
\end{eqnarray}
and
\begin{equation*}
\omega(t) =\left\{\begin{array}{ccc}
(1-\theta)^{-1} \big(2\hspace{0.05cm}t\hspace{0.05cm}
\|A\|_{ {\mathcal B}(\ell^2(\GG))}\big)^{\log_2(2-\theta)} & \textrm{if} & \theta\in (0, 1)\\
 &  &  \\
  \log_2 \big(1+2t \|A\|_{ {\mathcal B}(\ell^2(\GG))} \big)  & \textrm{if}  & \theta=1
\end{array}\right. .
\end{equation*}
\end{thm}

\begin{proof}
For all $t \in \big[\hspace{0.05cm}0,\hspace{0.05cm}\|A\|_{{\mathcal B}(\ell^2(\GG))}^{-1}\hspace{0.05cm}\big]$, we have
\begin{eqnarray}
\big\|e^{tA}\big\|_{\mathcal A}^q & \le &  \sum_{n=0}^\infty ~\Big(\frac{ t^n}{n!}\Big)^q
~\|A^n\cdot I\|_{{\mathcal A}}^{q} \nonumber \\
& \le &  M^q ~\sum_{n=0}^\infty ~\frac{ K_0^{nq}}{(n!)^q} \left(  \frac{\|A\|_{\mathcal A}}{\|A\|_{{\mathcal B}(\ell^2(\GG))}}\right)^{nq}.\label{exponentialstablesqw.lem.pf.eq1}
\end{eqnarray}
The first inequality follows from the  triangle inequality for $\|\cdot\|_{\mathcal A}^q$  and
 the Taylor expansion $e^{tA}=\sum_{n=0}^\infty \frac{t^nA^n}{n!}$,
 and  the second inequality holds due to the submultiplicative property (iv) in Definition \ref{def-q-banach}.
From this result, it follows that
\begin{equation} \label{exponentialstablesqw.lem.pf.eq2}
\big\|e^{tA}\big\|_{\mathcal A}^q ~\le~  M^q ~\Xi_{0} ~e^{\frac{\alpha q}{\|A\|_{{\mathcal B}(\ell^2(\GG))}}}~
 e^{-\alpha q  t}
\end{equation}
for all $t \in \big[\hspace{0.05cm}0,\hspace{0.05cm}\|A\|_{{\mathcal B}(\ell^2(\GG))}^{-1}\hspace{0.05cm}\big]$. From the following identity
\[ A^{-1}=-\int_0^\infty e^{tA} dt ~~~~ {\rm in } ~~ {\mathcal B}(\ell^2(\GG)),\]
Property (P2) in Definition \ref{def-proper} and the exponential stability
assumption  \eqref{es.def0}, one can conclude that
$\|A^{-1} \|_{{\mathcal B}(\ell^2(\GG))}\le \frac{E}{\alpha}$. Thus, we get
\begin{equation*}
\frac{\alpha}{\|A\|_{{\mathcal B}(\ell^2(\GG))}}~\le~ \frac{E}{\|A^{-1}\|_{{\mathcal B}(\ell^2(\GG))}\|A\|_{{\mathcal B}(\ell^2(\GG))}}~\le~ E.
\end{equation*}
This together with \eqref{exponentialstablesqw.lem.pf.eq2} implies that
\begin{equation} \label{exponentialstablesqw.lem.pf.eq3}
\big\|e^{tA}\big\|_{\mathcal A}^q ~\le~  M^q ~\Xi_{0}~ e^{qE} ~e^{-\alpha q t}
\end{equation}
for all $t \in \big[\hspace{0.05cm}0,\hspace{0.05cm}\|A\|_{{\mathcal B}(\ell^2(\GG))}^{-1}\hspace{0.05cm}\big]$. For every $t~> ~ \|A\|_{{\mathcal B}(\ell^2(\GG))}^{-1}$, let denote $m$ to be the smallest positive integer such that
\[2^{-m} t\in \big(0,~  \|A\|_{{\mathcal B}(\ell^2(\GG))}^{-1}\big],\]
or equivalently,
\begin{equation} \label{exponentialstablesqw.lem.pf.eq4}
 t \|A\|_{ {\mathcal B}(\ell^2(\GG))}~\le~  2^m~<~ 2 t \|A\|_{{\mathcal B}(\ell^2(\GG))}.
\end{equation}
By applying the differential norm property given in Definition \ref{def-diff} 
and the exponential stability assumption \eqref{es.def0} leads to
\begin{eqnarray}
\|e^{tA}\|_{\mathcal A}^q & \le &  2 D ~\|e^{\frac{t}{2}A}\|_{\mathcal A}^{q(2-\theta)}~
\|e^{\frac{t}{2}A}\|_{{\mathcal B}(\ell^2(\GG))}^{q\theta} \nonumber\\
& \le &  2DE^{q\theta}~
 e^{-q \theta  \alpha t/2} ~\|e^{\frac{t}{2}A}\|_{\mathcal A}^{q(2-\theta)}. ~~~~~~~~~~\label{exponentialstablesqw.lem.pf.eq5}
\end{eqnarray}
Then, it follows that
\begin{eqnarray}
\hspace{0cm} \|e^{tA}\|_{\mathcal A}^q & \le & \left(2DE^{q\theta}\right)^{1+ (2-\theta)} ~ e^{-\frac{\theta}{2}   \left(1+\frac{2-\theta}{2}\right) q \alpha t} 
 \|e^{\frac{tA}{4}}\|_{\mathcal A}^{q(2-\theta)^2 }\nonumber\\
& \le &  \cdots\nonumber\\
&  \leq & \Big(2DE^{q\theta}\Big)^{\sum_{j=0}^{m-1} (2-\theta)^j}e^{- \frac{\theta}{2} \left(\sum_{j=0}^{m-1}\left(\frac{2-\theta}{2}\right)^j \right) q \alpha t} ~\|e^{\frac{tA}{2^m}}\|_{{\mathcal A}}^{q(2-\theta)^m }\nonumber\\
& \leq &  M^q ~\Xi_{0}~ e^{qE}
\Big(2DE^{q\theta}\Big)^{\sum_{j=0}^{m-1} (2-\theta)^j} e^{-q\alpha t}
\end{eqnarray}
%
for all $t\ge  \|A\|_{{\mathcal B}(\ell^2(\GG))}^{-1}$.
In the above inequalities, the first three inequalities follow from
applying \eqref{exponentialstablesqw.lem.pf.eq5} repeatedly,
and the last inequality hold by \eqref{exponentialstablesqw.lem.pf.eq3}.
 This combines with \eqref{exponentialstablesqw.lem.pf.eq4} to show that the desired  exponential stability property \eqref{exponentialst ablesqw.lem.eq1}  in ${\mathcal A}$ holds for all $t \geq 0$.
\end{proof}

We should clarify that $C(t)$ in the right hand side of \eqref{exponentialst ablesqw.lem.eq1} has a sub-exponential growth for $0<\theta<1$ and a polynomial growth for $\theta=1$. Thus, the right hand side of \eqref{exponentialst ablesqw.lem.eq1} vanishes exponentially as $t$ goes to infinity, which implies that $A$ is exponentially stable in ${\mathcal A}$.

In the next step, we consider Lyapunov stability of linear systems that are defined over $q$-Banach algebras.

\section{Lyapunov Equations over $q$-Banach Algebras} \label{sec-V}
In this section, we  study structural properties of solutions of Lyapunov equations for linear systems  over $q$-Banach algebras. Our goal is to determine degrees of sparsity and spatial localization for the solution of the algebraic Lyapunov equation. We recall a well-known result about solving the algebraic Lyapunov equation in ${\mathcal B}(\ell^2(\GG))$; for more details see \cite[P. 76]{barnettbook70} and
\cite[Theorem 1]{bunce85}.
Suppose that $Q\in {\mathcal B}(\ell^2(\GG))$ is  strictly positive on $\ell^2(\GG)$, and
 $A\in {\mathcal B}(\ell^2(\GG))$  is exponentially stable on $\ell^2(\GG)$.
Then, there exists a unique  strictly positive  solution $P$ in  ${\mathcal B}(\ell^2(\GG))$ for
the  Lyapunov equation
\begin{equation}\label{lyapunov.equation}
AP+PA^*+Q=0,
\end{equation}
which has the following compact form
\begin{equation} \label{lyapunov.classical.lem.eq1}P~=~\int_0^\infty e^{tA^*} \hspace{0.05cm} Q \hspace{0.05cm} e^{tA}~ dt ~~~\textrm{in}~~~{\mathcal B}(\ell^2(\GG)).
\end{equation}

Our goal is to characterize sufficient conditions under which the unique solution of a Lyapunov equation with coefficients in a $q$-Banach algebra $\mathcal{A}$ also belongs to $\mathcal{A}$. Since $\mathcal{A}$ is not a Banach space for $0 < q <1$, we cannot  directly apply the explicit expression \eqref{lyapunov.classical.lem.eq1}   to tackle this problem. The reason is that continuous integrals, such as \eqref{lyapunov.classical.lem.eq1}, cannot be defined properly on $q$-Banach algebras  as it is not a Banach space. In order to overcome this challenge, first we introduce a series of matrices $\big\{ P_m \big\}_{m \geq 0}$ using the following iterative process
\begin{equation}  \label{lyapunov.thm.pf.eq5}
P_m ~=~ e^{A^* \hspace{0.05cm} \|A\|_{\mathcal{B}(\ell^2(\GG))}^{-1}}~P_{m-1} ~e^{A \hspace{0.05cm} \|A\|_{\mathcal{B}(\ell^2(\GG))}^{-1}}~+~P_{0}\end{equation}
for $m \geq 1$, with initial value
\begin{equation}\label{lyapunov.thm.pf.eq1}
P_{0}~=~\sum_{n=0}^\infty \hspace{0.05cm}\sum_{m=0}^\infty ~\frac{\|A\|_{{\mathcal B}(\ell^2(\GG))}^{-m-n-1}}{(m+n+1) \hspace{0.05cm} m! \hspace{0.05cm} n!} ~(A^*)^m \hspace{0.05cm} Q  \hspace{0.05cm} A^n.
\end{equation}
Then we apply the result of Theorem \ref{exponentialstablesqw.thm} on  exponential stability in a proper $q$-Banach algebra
to prove the convergence of the proposed iterative procedure in the $q$-Banach algebra. It can be shown that  $P_m \in {\mathcal A}$ for all $m\ge 0$ and that the series converges in ${\mathcal A}$. In fact, the limit of series $\big\{ P_m \big\}_{m \geq 0}$ as $m$ tends to infinity converges to the unique solution of \eqref{lyapunov.equation}.

\begin{thm}\label{lyapunov.thm}
For $ 0 < q \leq 1$, suppose that ${\mathcal A}$ is a proper $q$-Banach subalgebra of ${\mathcal B}(\ell^2(\GG))$.
Assume that $Q\in {\mathcal A}$ is  strictly positive on $\ell^2(\GG)$ and that
 $A\in {\mathcal A}$  is exponentially stable on $\ell^2(\GG)$. Then, the unique  strictly positive solution of
the  Lyapunov equation \eqref{lyapunov.equation} belongs to ${\mathcal A}$. Moreover, we have the following bound estimation
\begin{equation}
\|P\|_{\mathcal A}^q ~ \le ~ \Xi_{0}^{2}\hspace{0.07cm}\Pi_{0}~
\frac{ \|Q\|_{\mathcal A}^q}{\|A\|_{{\mathcal B}(\ell^2(\GG))}^q}
\label{bound-lyap}
\end{equation}
where
\[
\Pi_{0}=1+
K_{0}^{2q}  M^{2q}  \Xi_{0}^2 e^{2qE}
 ~\sum_{k=1}^\infty ~
  \left(2DE^{q\theta}
 \right)^{2 \omega_k} \exp \left(-\frac{2 \alpha  q k}{\|A\|_{{\mathcal B}(\ell^2(G))}}\right)
, \]
in which
 constants $E, \alpha$ are defined in \eqref{es.def0} and $M, D, K_{0}, \theta$ in Definitions \ref{def-q-banach}--\ref{def-proper}, and
\[
\omega_k~=~\left\{\begin{array}{ccc}
(1-\theta)^{-1} (2k)^{\log_2(2-\theta)} & {\rm if}  & \theta\in (0, 1) \\
 &  &  \\
\log_2(1+2k) & {\rm if} & \theta=1
\end{array}\right..
\]
\end{thm}

\begin{proof}
For an exponentially stable matrix $A\in {\mathcal A}$, let us consider $P_0$ given by \eqref{lyapunov.thm.pf.eq1}. Then $P_0\in {\mathcal A}$ and its $q$-norm can be  bounded from above as follows:
\begin{eqnarray}
 \|P_0\|_{\mathcal A}^q  & \le  & \sum_{m,n=0}^\infty \hspace{-.1cm}
\left(\frac{\|A\|_{{\mathcal B}(\ell^2)}^{-m-n-1}}{(m+n+1)\hspace{0.05cm} m! \hspace{0.05cm} n!} \right)^q \big \|(A^*)^m \hspace{0.05cm} Q \hspace{0.05cm} A^n \big\|_{\mathcal A}^q\nonumber\\
&  \le & \frac{1}{\|A\|_{{\mathcal B}(\ell^2)}^q}  \sum_{m,n=0}^\infty
\left(\frac{\|A\|_{{{\mathcal B}(\ell^2)}}^{-m-n}}{ m!\hspace{0.05cm}n!}\right)^q   K_0^{q(m+n)} \|A^*\|_{{\mathcal A}}^{qm}
\|A\|_{\mathcal A}^{qn} \|Q\|_{\mathcal A}^q\nonumber\\
&  = & \frac{ \|Q\|_{\mathcal A}^q}{\|A\|_{{\mathcal B}(\ell^2)}^q}~
\Xi_{0}^{2} ~ < ~\infty.\label{lyapunov.thm.pf.eq2}
\end{eqnarray}
Since ${\mathcal A}$ is a subalgebra of ${\mathcal B}(\ell^2(\GG))$ by \eqref{wiener.thm.eq4},
$P_0$ is a  bounded operator on $\ell^2(\GG)$. Moreover, it follows that
\begin{eqnarray}
P_0 & = & \int_0^{\|A\|_{\mathcal{B}(\ell^2)}^{-1}}
 \sum_{m,n=0}^\infty
  \frac{\left(t A^* \right)^m Q \left(tA \right)^n}{m! \hspace{0.05cm} n!}~  dt \nonumber \\
& = & \int_0^{\|A\|_{\mathcal{B}(\ell^2)}^{-1}} e^{tA^*} Q  e^{tA}
~dt  \quad {\rm in} \quad {\mathcal B}(\ell^2(\GG)).\label{lyapunov.thm.pf.eq3}
\end{eqnarray}
This implies that $P_0$ is strictly positive. Now, let us consider the  iterative process \eqref{lyapunov.thm.pf.eq5} for all $m\ge 1$  with initial value $P_0$. By induction,
\begin{equation}  \label{lyapunov.thm.pf.eq6}
P_m ~=~ \sum_{k=0}^m ~e^{k A^* \|A\|_{{\mathcal B}(\ell^2)}^{-1}} ~P_0 ~e^{ k A\|A\|_{{\mathcal B}(\ell^2)}^{-1}},
\end{equation}
for all $m\ge 0$. From  \eqref{lyapunov.thm.pf.eq6} and the result of Theorem  \ref{exponentialstablesqw.thm}, we can obtain the following inequalities:
 \begin{eqnarray}
\big\|P_{m+1}-P_m \big\|_{\mathcal A}^q &\le& K_0^{2q}  M^{2q}  \Xi_{0}^2 e^{2qE} \|P_0\|_{\mathcal A}^q
\left(2DE^{q\theta}
 \right)^{2 \omega_m} e^{-2 q\alpha m\|A\|_{\mathcal A}^{-1}}\end{eqnarray}
for all $m\ge 1$.
This  implies that
$P_m \in {\mathcal A}$ for all $m\ge 0$ and that the sequence  $\big\{ P_m \big\}_{m \geq 0}$ converges in ${\mathcal A}$. We denote by $P_\infty$  the limit of sequence $\big\{ P_m \big\}_{m \geq 0}$ as $m$ tends to infinity, i.e.,
\begin{equation*} P_\infty=\lim_{m\to \infty} P_m\in {\mathcal A}.\end{equation*}
From \eqref{lyapunov.thm.pf.eq3}
and \eqref{lyapunov.thm.pf.eq6}, it follows that
\[P_m=\int_0^{(m+1)\|A\|_{{\mathcal B}(\ell^2)}^{-1}} e^{tA^*} Q \hspace{0.05cm} e^{tA} ~dt
\quad {\rm in} ~~~ {\mathcal B}(\ell^2(\GG))\]
 for all $m\ge 0$. This together with the exponential stability  property \eqref{es.def} implies that the series $P_m$ for all $m\ge 1$ converges to $\int_0^\infty e^{tA^*} Q e^{tA} dt$ in ${\mathcal B}(\ell^2(\GG))$.
 Thus,  it follows that
\[\int_0^\infty e^{tA^*} Q e^{tA} dt=P_\infty\in {\mathcal A}.\]
Therefore, the unique solution $P_{\infty}$ of the Lyapunov equation \eqref{lyapunov.equation}
belongs to ${\mathcal A}$.
\end{proof}

The upper bound in \eqref{bound-lyap} gives an estimate to what degree the unique solution of the Lyapunov equation \eqref{lyapunov.equation} can be  sparsified and spatially localized. The result of this theorem can be particularly used to study the problem of disturbance propagation is spatially decaying system. It turns out that this problem boils down to evaluation of the corresponding controllability Gramian, which is the solution of a Lyapunov equation. The recent work \cite{Bamieh12} studies this problem for the class of spatially invariant systems over tori.

\section{Riccati Equations over $q$-Banach Algebras}\label{sec:riccati}
The main result of this paper is stated in this section where we consider solving algebraic Riccati equations related to linear-quadratic regulator (LQR) problems over $q$-Banach algebras. The following result is well-known and classic \cite{curtain-book, Bensoussan-Mitter93}. Suppose that $A, B, Q, R \in {\mathcal B}(\ell^2(\GG))$ and operators $Q$ and $R$ are positive and strictly positive on $\ell^{2}(\GG)$, respectively. If $(A,B)$ is exponentially stabilizable and $(A,Q^{1/2})$ is exponentially detectable, then the Riccati equation
\begin{equation}\label{ARE}
A^{*}X+XA-XBR^{-1}B^{*}X+Q=0
\end{equation}
has a unique strictly  positive  solution $X\in {\mathcal B}(\ell^2(\GG))$. Furthermore, the closed-loop matrix  $A_{X}=A+K$ 
is exponentially stable in ${\mathcal B}(\ell^2(\GG))$, where $K$ is  the LQR state feedback matrix given by
\begin{equation}
K=-R^{-1}B^{*}X.  \label{LQR-gain}
\end{equation}

For the closed-loop matrix $A_X=A-BR^{-1}B^{*}X$ there exist positive constants $E$ and $\alpha$ such that
\begin{equation}\label{riccati.thm.pf.eq1}
\big\|e^{t A_X}\big\|_{{\mathcal B}(\ell^2(\GG))}~\le~ E \hspace{0.05cm}e^{-\alpha t}
\end{equation}
for all $t\ge 0$. Let  $\Omega$ be the rectangle region in the complex plane with vertices
\begin{eqnarray}
 -\frac{\alpha}{2}&\pm& 2 \hspace{0.05cm} \|A_X\|_{{\mathcal B}(\ell^2(\GG))} \hspace{0.05cm} i,\label{rectan-region-1}\\
-2\hspace{0.05cm} \|A_X\|_{{\mathcal B}(\ell^2(\GG)))}&\pm& 2 \hspace{0.05cm} \|A_X\|_{{\mathcal B}(\ell^2(\GG)))} \hspace{0.05cm}i. \label{rectan-region-2}
\end{eqnarray}
The boundary of the rectangle region $\Omega$ is denoted by $\Gamma$. The following two lemmas are useful in the proof of the main results in Theorems \ref{riccati.thm} and \ref{thm-riccati-relaxed}.

\begin{lem}
Suppose that $A_{X}$ is exponentially stable in ${\mathcal B}(\ell^2(\GG))$ and \eqref{riccati.thm.pf.eq1} holds for some positive constants $E$ and $\alpha$. Then,
\begin{equation}  \label{stability-1}
\left\|(zI-A_X)^{-1}\right\|_{{\mathcal B}(\ell^2(\GG)))}~\le ~  \frac{2E}{\alpha}
\end{equation}
and
\begin{equation}  \label{stability-2}
\left\|(zI+A_X)^{-1}\right\|_{{\mathcal B}(\ell^2(\GG)))}~\le ~  \frac{E}{\alpha}
\end{equation}
for all $z\in \Gamma$, where $\Gamma$ is the boundary of the rectangle region $\Omega$ in the complex plane with vertices \eqref{rectan-region-1}-\eqref{rectan-region-2}.
\end{lem}
\begin{proof}
Let us consider the rectangle region in the complex plane with vertices \eqref{rectan-region-1}-\eqref{rectan-region-2} and boundary $\Gamma$.
 For every $z\in \Gamma$ with ${\rm Re} \{ z\}\ne -\frac{\alpha}{2}$, we have that \[|z|~\ge~ 2\|A_X\|_{{\mathcal B}(\ell^2(\GG)))}.\]
 This implies that
 $z I-A_X$ is invertible and
\begin{equation}
\left\|(zI-A_X)^{-1}\right\|_{{\mathcal B}(\ell^2(\GG)))}  ~\le~  \frac{E}{\alpha}.\label{ineq-1}
\end{equation}
For every $z\in \Gamma$ with ${\rm Re} \{z\}=-\frac{\alpha}{2}$, it follows that
\begin{equation}
 \left\|(zI-A_X)^{-1}\right\|_{{\mathcal B}(\ell^2(\GG)))}
~\le~   \frac{2E}{\alpha}. \label{ineq-2}
 \end{equation}
By combining inequalities \eqref{ineq-1} and \eqref{ineq-2}, we get
 \begin{equation}  \label{ineq-3}
\left\|(zI-A_X)^{-1}\right\|_{{\mathcal B}(\ell^2(\GG)))}~\le ~  \frac{2E}{\alpha}
 \end{equation}
for all $z\in \Gamma$. Through a similar argument, we can show that $z I+A_X^*$  is invertible for all $z\in \Gamma$ and
 \begin{equation}
 \left\|(z I + A_{X}^*)^{-1}\right\|_{{\mathcal B}(\ell^2(\GG)))} ~\le~     \frac{E}{\alpha}\label{ineq-4}
\end{equation}
for all $z\in \Gamma$.
\end{proof}

\begin{lem}\label{lem-hamiltonian}
For $ 0 < q \leq 1$, suppose that ${\mathcal A}$ is a proper $q$-Banach subalgebra of ${\mathcal B}(\ell^2(\GG))$, $A, B, Q, R \in {\mathcal A}$, and   $Q$ and $R$ are positive and strictly positive on $\ell^2(\GG)$, respectively. Assume that $X$ is the solution of \eqref{ARE} and the corresponding closed-loop matrix  $A_{X}$ is exponentially stable. Let us define the Hamiltonian operator of the Riccati equation \eqref{ARE} by
\begin{equation}  \label{riccati.thm.pf.eq6}
{\bf H} = \left[\begin{array}{cc}A & -BR^{-1}B^{*} \\-Q & -A^{*}\end{array}\right].
\end{equation}
If we define
\begin{equation} \label{riccati.thm.pf.eq9}
{\bf E}~:=~ \frac{1}{2 \pi i}~\int_{\Gamma}~ \big(z {\bf I}_2 - {\bf H}\big)^{-1} ~dz,
\end{equation}
then
\begin{equation}\label{em2a.property}
\mathbf{E} \in \mathcal{M}_2({\mathcal A}),\end{equation}
where $\mathcal{M}_2({\mathcal A})$ is the algebra of $2 \times 2$  matrices with entries in ${\mathcal A}$ and is defined in Appendix \ref{appendix-block-matrix}. Moreover, we have
\begin{equation}
{\bf E} ~=~
\left[\begin{array}{cc}I-ZX & Z \\X(I-ZX) & XZ\end{array}\right] \label{Eformula}
\end{equation}
for some operator $Z\in {\mathcal B}(\ell^2(\GG))$.
\end{lem}
\begin{proof}
We denote the block identity matrix by
$${\bf I}_2= \left[\begin{array}{cc}I & 0 \\ 0 & I\end{array}\right].$$
Through direct computations one can show that
\begin{eqnarray}
& & \hspace{-1.8cm} z {\bf I}_2 -{\bf H} ~ = ~  {\bf T}~  \left[\begin{array}{cc} z I- A_{X} & BR^{-1}B^{*} \\0 & z I+ A_{X}^{*}\end{array}\right]  {\bf T}^{-1}\label{riccati.thm.pf.eq7}
\end{eqnarray}
where $z\in \Gamma$ and
\[ {\bf T}=\left[\begin{array}{cc}I & 0 \\X & I\end{array}\right].\]
The above  identity together with
\eqref{stability-1} and
\eqref{stability-2}
imply that
$z{\bf I}_2 -{\bf H}$ is invertible in ${\mathcal B}(\ell^2(\GG)))$ for all $z\in \Gamma$.
Therefore it follows from Theorem \ref{wiener.thm} and Lemma \ref{propermatrix.lem}
that
\begin{equation} \label{riccati.thm.pf.eq18}
(z {\bf I}_2-{\bf H})^{-1}\in \mathcal{M}_2({\mathcal A}) \end{equation}
for all $z\in \Gamma$. Let us define
\begin{equation} \label{riccati.thm.pf.eq9}
{\bf E}~:=~ \frac{1}{2 \pi i}~\int_{\Gamma}~ \big(z {\bf I}_2 - {\bf H}\big)^{-1} ~dz.
\end{equation}
For every given $z_0\in \Gamma$, we have
\begin{eqnarray}
\big(z{\bf I}_2-{\bf H}\big)^{-1}  & = &   \big(z_0{\bf I}_2-{\bf H}\big)^{-1}\left( {\bf I}_2-(z_0-z)(z_0{\bf I}_2-{\bf H})^{-1}\right)\nonumber\\
& = & 
\big(z_0 {\bf I}_2-{\bf H}\big)^{-1} \sum_{n=0}^\infty (z_0-z)^n \left((z_0 {\bf I}_2-{\bf H})^{-1} \right)^n, \label{riccati.thm.pf.eq19}
\end{eqnarray}
where the series converges in $\mathcal{M}_2({\mathcal A})$ whenever
\[K_{0}^{\frac{1}{q}}~|z-z_0|~\left\|(z_0 {\bf I}_2-{\bf H})^{-1} \right\|_{\mathcal{M}_2({\mathcal A})}~<~1.\]
This is because according to \eqref{riccati.thm.pf.eq17} 
\begin{eqnarray}
& & \hspace{-1.1cm} \left\|(z_0-z)^n \left((z_0 {\bf I}_2-{\bf H})^{-1}\right)^n \right\|_{\mathcal{M}_2({\mathcal A})}^q ~\le~ |z_0-z|^{qn}~  K_{0}^{n-1} ~\left\|(z_0 {\bf I}_2-{\bf H})^{-1}\right\|_{\mathcal{M}_2({\mathcal A})}^{nq}. \label{ineq-H-2}
\end{eqnarray}
Since $\Gamma$ is a compact set, there exists finitely many points $z_i$ for $i=1,\ldots,N$ along with the four vertices of the rectangle region
such that  the collection of balls $B(z_{i},r_{i})$ with centers $z_{i}$ with  radii \[r_{i}=\frac{1}{2}\left(K_{0}^{\frac{1}{q}}~\left\|(z_i {\bf I}_2-{\bf H})^{-1}\right\|_{\mathcal{M}_2({\mathcal A})} \right)^{-1},\]
for all $i=1,\ldots,N$, is a covering of the boundary $\Gamma$. Therefore, we can write
\begin{eqnarray} \label{riccati.thm.pf.eq20}
{\bf E} & = & \sum_{i=1}^N ~\int_{\Gamma_i} ~(z{\bf I}_2-{\bf H})^{-1}~ dz\nonumber\\
& = & \sum_{i=1}^N ~\big(z_i {\bf I}_2-{\bf H}\big)^{-1}~  \sum_{n=0}^\infty \left(\int_{\Gamma_i} (z_i-z)^n dz\right)  \Big((z_i {\bf I}_2-{\bf H})^{-1}\Big)^n,
\end{eqnarray}
where
\[\Gamma_i ~\subset~ \Gamma ~\cap
 B(z_i, r_{i}),\]
for all $i=1,\ldots,N$ and $\big\{\Gamma_i:~ i=1,\ldots,N \big\}$ are disjoint covering of the boundary $\Gamma$. According to
\eqref{riccati.thm.pf.eq7}, \eqref{riccati.thm.pf.eq9} and \eqref{ineq-H-2}, we conclude that
\begin{equation}\label{em2a.property} \mathbf{E} \in \mathcal{M}_2({\mathcal A}).\end{equation}

From identity \eqref{riccati.thm.pf.eq7}, it follows that
\begin{eqnarray}
& & \hspace{-0.7cm} {\bf E} ~=~ \frac{1}{2\pi i}~
\int_{\Gamma} ~{\bf T}  \left[\begin{array}{cc} (z I- A_{X})^{-1}  &
 -(zI-A_X)^{-1} BR^{-1}B^{*} (zI+A_X^*)^{-1} \\0 & (z I+ A_{X}^{*})^{-1}\end{array}\right]   {\bf T}^{-1} dz \label{riccati.thm.pf.eq10}
\end{eqnarray}
belongs to ${\mathcal B}(\ell^2(\GG)))$
by
\eqref{stability-1} and
\eqref{stability-2}.
Recall that
 $\Gamma$ is the boundary of a  rectangle  region $\Omega$ such that the spectrum of $A_X$ is contained in $
 \Omega$ and the closure of $\Omega$ is contained in the open left half plane.
 Applying functional calculus  to \eqref{riccati.thm.pf.eq10} leads to
\begin{equation}
 {\bf E} ~=~  {\bf T}
\left[\begin{array}{cc}I & Z \\0 & 0\end{array}\right] {\bf T}^{-1}=
\left[\begin{array}{cc}I-ZX & Z \\X(I-ZX) & XZ\end{array}\right] \nonumber
\end{equation}
for some operator $Z\in {\mathcal B}(\ell^2(\GG)))$.
\end{proof}

From the result of Lemma \ref{lem-hamiltonian}, we can immediately conclude that
\begin{equation}\label{hamiltonian-sol}
I-ZX, Z, X(I-ZX), XZ \in  {\mathcal A}.
\end{equation}

Our goal is to prove that $X \in \mathcal{A}$. As we discussed earlier, a $q$-Banach algebra for $0 < q <1$ is not a Banach space. In the following two theorems, we apply Lemma \ref{lem-hamiltonian} along with several technical assumptions and develop a constructive proof based on finite covering of compact sets to show that the unique solution of an algebraic Riccati equation over a $q$-Banach algebra belongs to that $q$-Banach algebra.

\begin{thm}\label{riccati.thm}
For $ 0 < q \leq 1$, suppose that ${\mathcal A}$ is a proper $q$-Banach subalgebra of ${\mathcal B}(\ell^2(\GG))$, $A, B, Q, R \in {\mathcal A}$, and   $Q$ and $R$ are positive and strictly positive on $\ell^2(\GG)$, respectively. If we assume that (i) $(A,B)$ is exponentially stabilizable  and $(A,Q^{1/2})$ is exponentially detectable; and (ii) the dual Riccati equation
\begin{equation}
AY+YA^{*}-YQY+BR^{-1}B^{*}=0\label{filter-ARE}
\end{equation}
has a self-adjoint solution $Y \in {\mathcal B}(\ell^2(\GG))$ such that $I+YX$ is invertible in ${\mathcal B}(\ell^2(\GG))$, 
then the unique positive definite solution of the Riccati equation satisfies $X\in {\mathcal A}$ and $A_{X}$ is exponentially stable on $\mathcal{A}$.
\end{thm}

\begin{proof} We have ${\bf H}{\bf E}={\bf E}{\bf H}$, where ${\bf H}$ and ${\bf E}$ are defined in the statement of Lemma \ref{lem-hamiltonian}. It is straightforward to verify that $Z$ in \eqref{hamiltonian-sol} satisfies the Lypunov equation
\begin{equation} \label{ARE.Lyapunov.eq} A_XZ+ZA_X^*+BR^{-1}B^{*}=0.\end{equation}
From our assumption (ii) and exponential stability property of $A_X$ in ${\mathcal B}(\ell^2)$, the unique solution of  the Lyapunov equation \eqref{ARE.Lyapunov.eq} can be represented as (see \cite[Lemma 4.9]{curtain2006} for more details)
\begin{equation}
Z=Y(I+XY)^{-1},\label{Z-Y}
\end{equation}
where $Y=Y^{*} \in \mathcal{B}(\ell^{2})$ is a solution of the dual Riccati equation \eqref{filter-ARE}.  From \eqref{Z-Y}, we have
\begin{eqnarray}
I-ZX& = & I-Y(I+XY)^{-1}X \nonumber\\
& = & I-YX(I+YX)^{-1}\nonumber\\
& = & (I+YX)^{-1} . \label{I+YX}
\end{eqnarray}
According to our assumption (ii) and \eqref{I+YX}, it follows that $I-ZX$ is invertible in $ {\mathcal B}(\ell^{2}(\GG))$. From \eqref{hamiltonian-sol} and the inverse-closedness property of the $q$-Banach algebra ${\mathcal A}$ given in Theorem \ref{wiener.thm}, we can conclude that $(I-ZX)^{-1}\in \mathcal {A}$, which together with $X(I-ZX)\in \mathcal {A}$ in \eqref{hamiltonian-sol} proves $X \in \mathcal{A}$. From the closure properties of $\mathcal{A}$, it also follows that
\[ A_{X}=A-BR^{-1}B^{*}X \in \mathcal{A}.\]
Since $A_{X}$ is exponentially stable on $\mathcal{B}(\ell^{2}(\GG))$, we can conclude from Theorem  \ref{exponentialstablesqw.thm} that it is also exponentially stable on $\mathcal{A}$.
\end{proof}

One can relax the second condition in Theorem \ref{riccati.thm} by assuming that the linear system is approximately controllable. The following definition is taken from \cite[Definition 4.1.17]{curtain-book}.

\begin{defn}\label{def-approx-cont}
The linear system \eqref{linear-sys-1} is approximately controllable if it is possible to steer from the origin to within an $\epsilon$--neighborhood from every point in $\ell^{2}(\GG)$ for every arbitrary  $\epsilon >0$, i.e., if  $\mathfrak{R}$ the reachable subspace of $(A,B)$ is dense in $\ell^{2}(\GG)$, where
\begin{equation*}
\mathfrak{R}:=\Big\{z \in \ell^{2}(\GG) ~\Big| ~\exists ~T > 0,~u \in L^{2}\big([0,T];\ell^{2}(\GG)\big)~ \textrm{s.t.}~ z=\mathfrak{B}^T u ~\Big\},
\end{equation*}
and the controllability operator of $(A,B)$ is defined by
\[\mathfrak{B}^T u:=\int_{0}^{T} e^{A(T-\tau)}Bu(\tau) d\tau ~~~\textrm{in}~~~\mathcal{B}(\ell^{2}(\GG)). \]
\end{defn}

\begin{thm}\label{thm-riccati-relaxed}
For $ 0 < q \leq 1$, suppose that ${\mathcal A}$ is a proper $q$-Banach subalgebra of ${\mathcal B}(\ell^2(\GG))$, $A, B, Q, R \in {\mathcal A}$, and   $Q$ and $R$ are positive and strictly positive on $\ell^2(\GG)$, respectively. If we assume that $(A,B)$ is approximately controllable and  $(A,Q^{1/2})$ is exponentially detectable,
then the unique positive definite solution of the Riccati equation satisfies $X\in {\mathcal A}$ and $A_{X}$ is exponentially stable on $\mathcal{A}$.
\end{thm}

\begin{proof} According to \cite[Definition 4.1.17]{curtain-book} and Definition \ref{def-approx-cont}, we can also rewrite the reachable subspace of $(A,B)$ as
\begin{equation}
\mathfrak{R}=\bigcup_{T > 0} \mathcal{R}(\mathfrak{B}^T), \label{reachable-1}
\end{equation}
where $\mathcal{R}(.)$ is the range of an operator. From \cite[Lemma 4.1.6]{curtain-book}, it follows that
\begin{equation}
\mathcal{R}(\mathfrak{B}^T)=\mathcal{R}(\mathfrak{B}^T_{K})~~~\textrm{for all}~~~T > 0,\label{reachable-K}
\end{equation}
where $\mathfrak{B}^T_{K}$ is the controllability operator of $(A+BK,B)$. From \eqref{reachable-1} and \eqref{reachable-K}, we can conclude that the reachable subspace of $(A,B)$ is equal to $(A+BK,B)$ for any bounded feedback $K$. Therefore, $(A,B)$ is approximately controllable if and only if $(A+BK,B)$ is for any bounded feedback $K$. This implies that $(A_{X},B)$ is also approximately controllable. We have ${\bf H}{\bf E}={\bf E}{\bf H}$, where ${\bf H}$ and ${\bf E}$ are defined in the statement of Lemma \ref{lem-hamiltonian}. The matrix $Z$ in \eqref{hamiltonian-sol} satisfies the Lyapunov equation \eqref{ARE.Lyapunov.eq}. As $A_X$ is exponentially stable in  ${\mathcal B}(\ell^2)$, it follows that
\begin{eqnarray}
Z & = & -\int_0^\infty ~\frac{d}{dt} \left(e^{tA_X} Z e^{A_X^*} \right) dt \nonumber \\
& = &  -\int_0^\infty   e^{tA_X}~ \big(A_X Z+Z A_X^* \big)~ e^{t A_X^*} ~dt \nonumber \\
& = & \int_0^\infty ~e^{tA_X} ~BR^{-1}B^{*}~ e^{t A_X^*}~  dt~~~\textrm{in}~~~\mathcal{B}(\ell^{2}(\GG)).\nonumber
\end{eqnarray}

Since $(A_{X},B)$ is approximately controllable, according to \cite[Theorem 4.1.22]{curtain-book} we can conclude that $Z$ is the unique strictly positive solution of \eqref{ARE.Lyapunov.eq}. In the next step, we prove that $Z$ belongs to the $q$-Banach algebra ${\mathcal A}$. According to Lemma 9.2, we have $\mathbf{E} \in \mathcal{M}_2({\mathcal A})$.
From this result and \eqref{Eformula}, it immediately follows that
\begin{equation}\label{Z-XZ} Z,\ XZ\in  {\mathcal A}.\end{equation}
Since $Z$ is strictly positive, we get that $Z^{-1}\in {\mathcal A}$ by the inverse-closedness of the $q$-Banach algebra ${\mathcal A}$. This together with \eqref{Z-XZ} prove  that the solution $X$ of the algebraic Riccati equation \eqref{ARE} belongs to
${\mathcal A}$. Based on a similar argument in the proof of Theorem \ref{riccati.thm}, we can conclude that $A_{X}$ is exponentially stable on $\mathcal{A}$.
\end{proof}

We remark that the result of Theorem \ref{thm-riccati-relaxed} does not require the existence of a solution of the corresponding filter Riccati equation.

\begin{cor}\label{LQR-gain-cor}
Suppose that the assumptions of Theorem \ref{riccati.thm} or Theorem \ref{thm-riccati-relaxed} hold. Then, the LQR state feedback gain \eqref{LQR-gain} belongs to $\mathcal{A}$.
\end{cor}
\begin{proof}
According to Theorem \ref{riccati.thm} or Theorem \ref{thm-riccati-relaxed}, $X \in \mathcal{A}$. Since $\mathcal{A}$ is closed under matrix multiplication and taking inverse, it follows that $K=-R^{-1}B^{*}X \in \mathcal{A}$.
\end{proof}

\begin{rem} The proof of Theorem \ref{riccati.thm} is constructive which enables us to obtain an upper bound for the $q$-norm of $X$ over each covering set. By combining these upper bounds, one can only calculate a conservative upper bound for the $q$-norm of $X$. For this reason, we do not present this conservative upper bound.
\end{rem}

\begin{rem}
The conclusions in  Theorems \ref{lyapunov.thm} and \ref{riccati.thm} hold for $\mathcal{S}_{q,w}(\GG)$ for all  $0 < q \leq 1$. For $1 \leq q \leq \infty$, $\mathcal{S}_{q,w}(\GG)$ is a Banach algebra. Our proof does not depend on whether  the underlying space is a Banach space or not. Therefore, our proofs are still valid for all exponents $0 < q \leq \infty$.
\end{rem}

\begin{rem}
Our proposed methodology in this paper can be applied to analyze other optimal control design problems with quadratic performance criteria, such as $\mathcal{H}_{2}$ and $\mathcal{H}_{\infty}$ problems \cite{curtain-book}. The solutions of these optimal control problems usually involve solving two operator Riccati equations.
\end{rem}

\section{Fundamental Limits on the Best Achievable Degrees of Sparsity and Spatial Localization}\label{sec-VI}

The results of Theorem \ref{schur-cor} and Corollary \ref{LQR-gain-cor} imply that the LQR state feedback gain satisfies
\[ K \in \mathcal{S}_{q,w}(\GG) ~~~\textrm{for all}~~~ \ 0<q\le 1. \]
This result states that the LQR state feedback gain $K=[K_{ij}]_{i,j \in \GG}$ is spatially localized, i.e.,
\begin{equation}
|K_{ij}|~\le~C_{0}\hspace{0.03cm} w(i,j)^{-1},
\end{equation}
where $C_{0}=\|K\|_{\mathcal{S}_{q,w}(\GG)}$ is a finite number. This also asserts that the underlying information structure of the optimal control law is sparse in space and each local controller needs to receive information only from some neighboring subsystems. In the rest of this section in order to present our results in more explicit and sensible forms, we will limit our analysis to matrices that are defined on $\GG=\ZZ$ endowed with quasi-distance function $\rho(i,j)=|i-j|$.

In order to characterize a fundamental limits on interplay between stability margins and optimal performance loss in spatially distributed systems, we approximate the corresponding LQR feedback gain  $K$ by a sparse feedback gain $K^{\mathfrak{T}}$ for a given truncation length $\mathfrak{T} > 0$ as follows
\begin{equation}\label{K-trunc}
\big(K^{\mathfrak{T}} \big)_{ij} = \left\{
                        \begin{array}{ccc}
                          K_{ij} & \textrm{if} & |i-j| \leq \mathfrak{T} \\
                          0 & \textrm{if} & |i-j| > \mathfrak{T} \\
                        \end{array}
                      \right..
\end{equation}

\begin{thm}\label{thm-asymp-K}
For $0 < q \leq 1$, suppose that $K \in \mathcal{S}_{q,w}(\ZZ)$ and  $K^{\mathfrak{T}}$ is defined by \eqref{K-trunc}. Then,
\begin{equation}\label{trunc-bound}
\|K-K^{\mathfrak{T}}\|_{\mathcal{B}(\ell^{2}(\ZZ))}~\leq ~C_{0} \hspace{0.03cm} w(\mathfrak{T})^{-1}
\end{equation}
where $C_{0}=\|K\|_{\mathcal{S}_{q,w}(\ZZ)}$ and $w(\mathfrak{T})=\inf_{|i-j|>\mathfrak{T}} w(i,j)$.
\end{thm}

\begin{proof}
We apply the Schur test to estimate the operator norm in $\mathcal{B}(\ell^2(\ZZ))$, i.e.,
\begin{eqnarray}
\hspace{-.7cm} \| K - K^{\mathfrak{T}} \|^2_{\mathcal{B}(\ell^2(\ZZ))} &\leq & \Big(\sup_{i \in \ZZ}~ \sum_{j \in \ZZ}~\big|K_{ij} - (K^{\mathfrak{T}})_{ij}\big| \Big)   \Big( \sup_{j \in \ZZ}~ \sum_{i \in \ZZ}~\big|K_{ij} - (K^{\mathfrak{T}})_{ij}\big|  \Big). \nonumber
\end{eqnarray}
First, let us consider the following term
\begin{eqnarray}
 \sup_{i\in \ZZ} ~\sum_{j\in \ZZ}~\big|K_{ij} - (K^{\mathfrak{T}})_{ij}\big| & = & \sup_{i\in \ZZ} \sum_{|i-j| > \mathfrak{T}}\big| K_{ij} \big|  \nonumber \\
&  = &  \sup_{i\in \ZZ}~ \sum_{|i-j| > \mathfrak{T}}\big|K_{ij}\big| w(i, j)  w(i, j)^{-1} \nonumber \\
&  \leq & \Big( \sup_{i \in \ZZ} \sum_{|i-j| > \mathfrak{T}} \big|K_{ij}\big| w(i, j) \Big)  w(\mathfrak{T})^{-1} \nonumber\\
&  \leq & C_{0} \hspace{0.03cm} w(\mathfrak{T})^{-1}.
\end{eqnarray}
The last inequality holds due to the monotonic inequality
\[ \|z\|_{1} \leq \|z\|_{q}\]
for all $z \in \ell^{q}(\ZZ)$ and $0 < q \leq 1$.   Likewise, by interchanging indices $i$ and $j$ a similar bound can be obtained. Therefore, it concludes that
\[  \| K - K^{\mathfrak{T}} \|_{\mathcal{B}(\ell^2(\ZZ))} ~\leq~ C_{0} ~ w(\mathfrak{T})^{-1}. \]
\end{proof}
For the class of sub-exponential coupling weight functions with parameters $\sigma >0$ and $\delta \in (0,1)$, the inequality \eqref{trunc-bound} becomes
\[  \| K - K^{\mathfrak{T}} \|_{\mathcal{B}(\ell^{2}(\ZZ))} ~\leq~C_{\textrm{e}}~ e^{-\left(\frac{\mathfrak{T}}{\sigma}\right)^{\delta}} ,\]
where $C_{\textrm{e}}=\|K\|_{\mathcal{S}_{q,e_{\sigma,\delta}}(\ZZ)}$.
Similarly, the error bound for the class polynomial coupling weight functions with parameters $\alpha, \sigma>0$ is
\[ \| K - K^{\mathfrak{T}} \|_{\mathcal{B}(\ell^{2}(\ZZ))} ~\leq~C_{\pi}\left(\frac{\mathfrak{T}}{\sigma}\right)^{-\alpha},\]
where $C_{\pi}=\|K\|_{\mathcal{S}_{q,\pi_{\alpha,\sigma}}(\ZZ)}$.

The inequality \eqref{trunc-bound} in Theorem \ref{thm-asymp-K} implies that for spatially decaying state feedback gains, the truncation tail $K-K^{\mathfrak{T}}$ can be made arbitrarily small as $\mathfrak{T}$ gets large. This property enables us to characterize stabilizing truncated state feedback gains by using small-gain stability argument.  Let us consider the truncated closed-loop system
\begin{equation}
\dot{x}=(A-BK^{\mathfrak{T}})x.
\end{equation}
One can decompose this system as two subsystems
\begin{eqnarray}
\dot{x}&=&(A-BK)x+w, \label{small-gain-closed-loop-1}\\
w&=&B(K-K^{\mathfrak{T}})x. \label{small-gain-closed-loop-2}
\end{eqnarray}
Since $K$ is the LQR feedback gain, the LQR closed-loop operator $A-BK$ in $\eqref{small-gain-closed-loop-1}$ is exponentially stable. If $B \in \mathcal{B}(\ell^{2}(\GG))$, we can apply small-gain theorem and characterize a fundamental limit in the form of a lower bound for stabilizing truncation lengths for different admissible coupling weight functions. For the class of sub-exponential coupling weight functions with parameters $\sigma > 0$ and $\delta \in (0,1)$, the truncated feedback gain $K^{\mathfrak{T}}$ is exponentially stabilizing if the truncation length $\mathfrak{T}$ satisfies the following inequality
\begin{equation}
\mathfrak{T} ~>~\mathfrak{T}_s, \label{stabilizing-truncation-length}
\end{equation}
where
\begin{equation}
\mathfrak{T}_s=\sigma \left(\log\Big(C_{\textrm{e}}\hspace{0.05cm} \|B\|_{\mathcal{B}(\ell^2(\ZZ))}\sup_{\mathrm{Re}(s) > 0} \sigma_{\max}(G(s)) \Big)\right)^{\frac{1}{\delta}} \label{stab-trunc-length}
\end{equation}
and $G(s)=(sI-(A-BK))^{-1}$ is the transfer function of the closed-loop system. The stabilizing truncation length for the class of polynomial coupling weight functions with parameters $\alpha, \sigma>0$ is given by
\begin{equation*}
\mathfrak{T}_{s}~=~\sigma \left( C_{\pi}\hspace{0.05cm} \|B\|_{\mathcal{B}(\ell^2(\ZZ))}\sup_{\mathrm{Re}(s) > 0} \sigma_{\max}(G(s))\right)^{\frac{1}{\alpha}}.
\end{equation*}

A fundamental tradeoff emerges between truncation length and performance loss. Let us consider the following quadratic cost functional
\begin{equation}
J(x_{0},K)=\int_{0}^{\infty} \left(x^{*}Q x + x^{*}K^{*}RKx  \right) dt,
\end{equation}
where $x_{0}$ is the initial condition of the linear system and $K$ an exponentially stabilizing state feedback gain. It is straightforward to show that
\begin{equation}\label{cost-X}
J(x_{0},K)=x_{0}^{*} X x_{0},
\end{equation}
in which $X$ is the unique strictly positive definite solution of the Lyapunov equation
\begin{equation}\label{lyap-closed-loop}
(A-BK)^{*}X+X(A-BK)+Q+K^{*}RK=0.
\end{equation}

We define the performance loss measure by the following quantity
\begin{equation}\label{perfor-loss}
\Pi_{K}(\mathfrak{T}, x_0)~=~J(x_0,K^{\mathfrak{T}})-J(x_0,K).
\end{equation}
One of the important remaining problems is to study the asymptotic behavior of the performance loss $\Pi_{K}(\mathfrak{T}, x_0)$ as $\mathfrak{T}$ tends to infinity. This problem is out of scope of this paper and will be addresses in our future works.

\section{Near-Optimal Degrees of Sparsity and  Spatial Localization}\label{sec-near-ideal-meas}
In the next step, our goal is to propose a method to compute a value for parameter $0 < q < 1$ such that $\mathcal{S}_{q,w}$--measure approximates $\mathcal{S}_{0,1}$--measure in probability. In this subsection in order to present our results in more explicit and sensible forms, we will limit our analysis to the class of sub-exponentially decaying matrices that are defined on $\GG=\ZZ$.

\begin{defn}
For a given truncation threshold $\epsilon >0$ and matrix $K$, the threshold  matrix of $K$ is denoted by
$K_\epsilon$ and defined by setting $(K_{\epsilon})_{ij}=0$ if $|K_{ij}| <\epsilon$ and  $(K_{\epsilon})_{ij}=K_{ij}$ otherwise.
\end{defn}

In order to present our results in more explicit and sensible forms, only in this section, we limit our focus to the class of  sub-exponentially decaying random matrices of the form
\begin{equation} 
\mathcal{R}_{\sigma, \delta}(\ZZ)= \left\{\left. K=\left[r_{ij}\hspace{0.05cm}\hspace{0.03cm} e^{-\left(\frac{|i-j|}{\sigma}\right)^{\delta}}\right]_{i,j\in \ZZ} \right|r_{ij} \sim \mathbf{U}(-1,1) \right\} \nonumber
\end{equation}
for some given parameters $\sigma>0$ and $ \delta \in (0,1)$.  The coefficients $r_{ij}$ are drawn from the continuous uniform distribution $\mathbf{U}(-1,1)$. It is assumed that the underlying spatial domain is $\ZZ$ and that the spatial distance between node $i$ and $j$ is measured by $|i-j|$. The corresponding admissible coupling weight function  for this class of spatially decaying matrices is given by
 \begin{equation}\label{onesubexponentialweight.def}
e_{\sigma', \delta}:=\big[e_{\sigma', \delta}(i,j) \big]_{i,j\in \ZZ}=\left[e^{\left(\frac{|i-j|}{\sigma'}\right)^\delta} \right]_{i,j\in \ZZ}\end{equation}
for some $\sigma' > \sigma$.
\vspace{0.1cm}
\begin{defn}
For a given truncation threshold $0 < \epsilon < 1$, the sparsity indicator function for the class of random sub-exponentially decaying matrices $\mathcal{R}_{\sigma, \delta}(\ZZ)$ is defined by
\begin{equation}
\Psi_{K, w}(q, \epsilon)~:=~ \frac{\|K\|_{\mathcal{S}_{q, w}}^{q} }{2 \hspace{0.05cm}\big \lfloor \sigma \sqrt[\delta]{\ln \epsilon^{-1}} \big \rfloor +1},\label{sparsity-indicator-fcn}
\end{equation}
where $\lfloor . \rfloor$ is the floor function.
\end{defn}

The sparsity indicator function provides a reasonable criterion for calculating proper values for exponent $q$ and parameters in the weight function in order to measure sparsity of matrices in $\mathcal{R}_{\sigma, \delta}(\ZZ)$ using $\mathcal{S}_{q, w}$--measure. This is simply because of the following inequality that shows that the value of the ${\mathcal{S}_{0,1}}$--measure of the threshold matrix of a matrix $K \in \mathcal{R}_{\sigma, \delta}(\ZZ)$ can be upper bounded by
\begin{equation}
\| K_{\epsilon}\|_{\mathcal{S}_{0,1}} ~\leq~2 \hspace{0.05cm}\big \lfloor \sigma \sqrt[\delta]{\ln \epsilon^{-1}} \big \rfloor +1.
\end{equation}


\begin{thm}\label{indicatorsubexponential.thm} Suppose that $w_0 \equiv 1$ is the trivial weight function and parameters $\sigma>0$ and $\delta\in (0,1)$ are given. Let us define parameter $\beta=q \ln \epsilon^{-1}$. For every $K \in \mathcal{R}_{\sigma, \delta}(\ZZ)$, the sparsity indicator function $\Psi_{K, w_0}(q, \epsilon)$ converges to
$\gamma:=\beta^{-s} \Gamma(s+1)$
in probability
as the truncation threshold $\epsilon$ tends to zero,
 i.e.,
\begin{equation}\label{subexponential.limit}
\lim_{\epsilon\to 0^{+}}~\mathbb{P}\Big\{~\big|\Psi_{K, w_0}(q, \epsilon)-\gamma \big| < \epsilon_0 ~\Big\}~=~1
\end{equation}
for every $\epsilon_0 >0$,
where $s=\delta^{-1}$ and $\Gamma(s):=\int_0^\infty t^{s-1} e^{-t} dt$ is the Gamma function.
\end{thm}
\begin{proof}
The exponent $q > 0$ can be related to another auxiliary variable $\beta > 0$ using the following equation
\begin{equation} \label{qsubexponential} \epsilon^q= e^{-\beta},
\end{equation}
from which we get $\beta=q \ln \epsilon^{-1}$. For every $K \in \mathcal{R}_{\sigma, \delta}(\ZZ)$, one can verify that
\begin{eqnarray}
\Psi_{K, w_0}(q, \epsilon)
 & \le &  \frac{1}{2 \hspace{0.05cm}\big \lfloor \sigma \sqrt[\delta]{\ln \epsilon^{-1}} \big \rfloor +1}~
\sum_{k\in \ZZ}~ e^{-q \left|\frac{k}{\sigma}\right|^\delta} \nonumber \\
& = & \frac{2\sigma \sqrt[\delta]{\ln \epsilon^{-1}}}{2 \hspace{0.05cm}\big \lfloor \sigma \sqrt[\delta]{\ln \epsilon^{-1}} \big \rfloor +1}~ \frac{1}{2\sigma \sqrt[\delta]{\ln \epsilon^{-1}}} ~ \sum_{k\in \ZZ}
 e^{- \beta  \left| \frac{k}{\sigma \sqrt[\delta]{\ln \epsilon^{-1}}} \right|^{\delta}}.\label{qsubexponential.comp1}
\end{eqnarray}
It is straightforward to show that
\begin{equation}
\lim_{\epsilon \to 0} ~\frac{2\sigma \sqrt[\delta]{\ln \epsilon^{-1}}}{2 \hspace{0.05cm}\big \lfloor \sigma \sqrt[\delta]{\ln \epsilon^{-1}} \big \rfloor +1} ~=~1, \label{frac-limit-1}
\end{equation}
and
\begin{eqnarray}
\hspace{-1.1cm} & & \frac{1}{2\sigma \sqrt[\delta]{\ln \epsilon^{-1}}} ~\sum_{k\in \ZZ}
 e^{- \beta  \left| \frac{k}{\sigma \sqrt[\delta]{\ln \epsilon^{-1}}} \right|^{\delta}}\longrightarrow   \frac{1}{2}  \int_{-\infty}^\infty e^{-\beta |t|^\delta} dt=\frac{1}{\delta}\int_0^\infty  e^{-\beta u} u^{\frac{1}{\delta}-1} du=\gamma
   \label{frac-limit-2}
\end{eqnarray}
as $\epsilon\to 0$. Thus, 
\begin{equation}
\lim_{\epsilon\to 0}
\Psi_{K, w_0}(q, \epsilon)~\leq~ \int_0^\infty e^{-\beta |t|^\delta} dt=\gamma.
\label{ineq-limit}
\end{equation}
 On the other hand, the expected value of the sparsity indicator function is lower bounded as follows
\begin{eqnarray} \label{qsubexponential.comp2}
\mathbb{E}\big[ \Psi_{K, w_0}(q, \epsilon)\big] & \ge &  \frac{1}{2 \hspace{0.05cm}\big \lfloor \sigma \sqrt[\delta]{\ln \epsilon^{-1}} \big \rfloor +1}~ \sum_{k\in \ZZ}~ \mathbb{E}\big[|r_{0j}|^q\big] ~ e^{-q \frac{|0-k|^\delta}{\sigma^\delta}}\nonumber\\
&   =  & \frac{1}{1+q}~\frac{1}{2 \hspace{0.05cm}\big \lfloor \sigma \sqrt[\delta]{\ln \epsilon^{-1}} \big \rfloor +1}~ \sum_{k\in \ZZ}
 e^{- \beta  \left| \frac{k}{\sigma \sqrt[\delta]{\ln \epsilon^{-1}}} \right|^{\delta}}. \nonumber
\end{eqnarray}
From \eqref{frac-limit-1} and \eqref{frac-limit-2}, we have that
\begin{equation}
\lim_{\epsilon\to 0}
\mathbb{E}\big[ \Psi_{K, w_0}(q, \epsilon)\big] ~\geq~ \int_0^\infty e^{-\beta |t|^\delta} dt =\gamma.
\label{expected-ineq}
\end{equation}
In \eqref{expected-ineq}, if we keep the auxiliary parameter $\beta$ fixed, then $q \to 0$ whenever $\epsilon \to 0$.
The two inequalities in limits \eqref{ineq-limit} and \eqref{expected-ineq} prove \eqref{subexponential.limit} which shows the convergence  of the sparsity indicator function $\Psi_{K, w_0}(q, \epsilon)$  in probability.
\end{proof}

The result of Theorem \ref{indicatorsubexponential.thm} asserts that
the  sparsity indicator functions associated with the class of sub-exponentially decaying random matrices $\mathcal{R}_{\sigma, \delta}(\ZZ)$  can be made arbitrarily close to a number close to one (i.e., $\gamma=1$ or $\beta= \Gamma(\frac{1+\delta}{\delta})^{\delta}$) in probability when we use trivial weight function $w_0$ (with constant value $1$). It turns out that one can select the exponent $q$ and the nontrivial weight function $w$ simultaneously to measure sparsity using the weighted  ${\mathcal S}_{q, w}$--measure with guaranteed convergence properties.
\begin{cor}
Consider the coupling weight function   \eqref{onesubexponentialweight.def} with parameter $\sigma'= \frac{\sigma }{\eta}$ for some $\eta, \delta \in (0,1)$ and $\sigma>0$. Let us define parameter $\beta = q \ln \epsilon^{-1}$. Then, for every $K \in \mathcal{R}_{\sigma, \delta}(\ZZ)$ we have
\begin{equation}\label{prob-limit}
\lim_{\epsilon\to 0^{+}}~\mathbb{P}\Big\{~\big|\Psi_{K, e_{\sigma', \delta}}(q, \epsilon)-\gamma' \big| < \epsilon_0 ~\Big\}~=~1,
\end{equation}
where
\[\gamma':=(1-\eta^\delta)^{\frac{1}{\delta}}~ \beta^{\frac{1}{\delta}} ~\Gamma\left(\frac{1+\delta}{\delta}\right).\]\end{cor}

\begin{rem}
Similar results can be obtained for polynomial weight functions.
\end{rem}
\vspace{0.1cm}

 {\it Algorithm for computing a near-optimal value for exponent $q$:}  There is an inherent algorithm in the result of Theorem \ref{indicatorsubexponential.thm} that provides us with a roadmap to compute a near-ideal truncation length for a given spatially decaying matrix. Suppose that we are given a sub-exponentialy decaying random matrices in $\mathcal{R}_{\sigma,\delta}(\ZZ)$. For a given truncation length $\mathfrak{T} \geq 1$, one can compute a conservative value for parameter $\epsilon$ using the following equation
\begin{equation}
\epsilon=e^{-\left(\frac{\mathfrak{T}}{\sigma} \right)^{\delta}}. \label{epsilon-value}
\end{equation}
By fixing the value of parameter $\beta$, one can compute a corresponding exponent $q$ using the value of $\epsilon$ from \eqref{epsilon-value} as well as parameter $\gamma$. Knowing all these parameters enable us to compute the value of the sparsity indicator function using its definition in  \eqref{sparsity-indicator-fcn}. An iterative procedure can be used to obtain the corresponding values of the sparsity indicator function for all values of $\mathfrak{T} \geq 1$ by repeating the above mentioned steps. Let us denote by $\mathfrak{T}_{\textrm{near-ideal}}$ the value of the truncation length beyond which the value of the sparsity indicator function converges to a number close to $\gamma$ with an acceptable error bound. The existence of such a truncation length for a given matrix in $\mathcal{R}_{\sigma,\delta}(\ZZ)$ implies that there exists an exponent $0 < q_{\textrm{near-ideal}}<1$ such that the $\mathcal{S}_{q,w}$--measure of the matrix converges to a number close to one in probability for all $0 < q < q_{\textrm{near-ideal}}$. According to fundamental limit \eqref{stabilizing-truncation-length}, a near-optimal truncation length for a given matrix in $\mathcal{R}_{\sigma,\delta}(\ZZ)$ should satisfy the following inequality
\begin{equation}
\mathfrak{T}_{\textrm{near-optimal}}\geq\max \big\{\mathfrak{T}_{s}, \mathfrak{T}_{\textrm{near-ideal}} \big\}.
\end{equation}
where $\mathfrak{T}_{s}$ is given by \eqref{stab-trunc-length}.

\section{Discussion and Conclusion}\label{sec-dis-concl}

\subsection{Finite-Dimensional Spatially Distributed Systems}
One can lift finite-dimensional spatially decaying systems to infinite-dimensional spatially decaying systems by defining an extension of a finite-dimensional matrix to an infinite-dimensional matrix. Suppose that $A,B,C,D \in \RR^{n \times n}$ are state-space matrices of a linear time-invariant system and $Q,R \in \RR^{n \times n}$ are the weight matrices in the LQR problem. The infinite-dimensional extension of a matrix can be defined using the following operation:
\begin{eqnarray}
& & \hspace{-.8cm} A, B,C,D  \longmapsto  \left[\begin{array}{cc}A & 0 \\0 & -I \end{array}\right], \left[\begin{array}{cc}B & 0 \\0 & I \end{array}\right],\left[\begin{array}{cc}C & 0 \\0 & I \end{array}\right],\left[\begin{array}{cc}D & 0 \\0 & I \end{array}\right], \nonumber \\
& & \hspace{-0.8cm} Q,R  \longmapsto  \left[\begin{array}{cc}Q & 0 \\0 & I \end{array}\right],\left[\begin{array}{cc}R & 0 \\0 & I \end{array}\right], \nonumber
\end{eqnarray}
where $I$ and $0$ are the identity and zero operators in $\mathcal{B}(\ell^{2}(\GG))$. The lifted operators have specific block diagonal structures which enable us to show that the unique solution of the corresponding Lyapunov equation \eqref{lyapunov.equation} takes the following block diagonal form
\begin{equation}
\left[\begin{array}{cc}P & 0 \\0 & \frac{1}{2}I \end{array}\right], \label{sol-lyap}
\end{equation}
where $P$ is the solution of the corresponding finite-dimensional Lyapunov equation.  The unique solution of the corresponding algebraic Riccati equation \eqref{ARE} admits the following block diagonal form
\begin{equation}
\left[\begin{array}{cc}X & 0 \\0 & I \end{array}\right],\label{sol-riccati}
\end{equation}
where $X$ is the unique solution of the corresponding finite-dimensional Riccati equation. Therefore, the results of Sections \ref{sec:riccati} and \ref{sec-V} can be applied to the lifted finite-dimensional linear systems. Due to the specific block structures of solutions \eqref{sol-lyap} and \eqref{sol-riccati}, one cannot obtain accurate spatial decay rates for these solution based on our present results. However, extensive simulation results suggest that similar sparsity and spatial localization properties should hold for finite-dimensional spatially decaying systems.

\subsection{Spatial Truncation and Performance Bounds}

Our proposed methodology in this paper provides  quantitative measures to determine degree of sparsity and spatial localization for a large class of spatially decaying systems and their LQR feedback controllers. This information necessitates development of  algorithmic methods to construct localized models with quantitative estimates, localized feedback control laws with quantitative performance bounds, and feedback control laws with sparse information structures. In Section \ref{sec-VI}, we showed that for a given LQR feedback controller there exist a class of spatially localized feedback controllers that can be calculated by direct spatial truncation of the LQR solution. One remaining important problem is to investigate the asymptotic behavior of the performance index $\Pi_{K}(\mathfrak{T}, x_0)$ as $\mathfrak{T}$ tends to infinity. Moreover, it is an interesting and open question to apply  finite-section approximation techniques \cite{Grochenig2010} in order to reduce the control design complexity and quantify the convergence rate of the finite-section approximations for spatially decaying systems.

\subsection{Optimization and Regularization To Promote Sparsity}
The inherent spatial decay property of the solution of LQR problem for spatially decaying systems suggests that searching for sparse linear-quadratic state feedback controllers should naturally have lower computational complexity compared to general spatially distributed systems. Our analysis in Section  \ref{sec-near-ideal-meas} provides a pathway to quantify what degrees of sparsity and spatial localization one should expect by solving $\mathcal{S}_{0,1}/\mathcal{S}_{q,w}$-- regularization methods (for $0 < q \le 1$) in order to design near-optimal sparse state feedback controllers. The $\mathcal{S}_{q,w}$-regularized optimal control problems (for $0 < q < 1$) are much more broadly applicable with respect to $\ell^1$-based relaxations, which results in $\mathcal{S}_{1,w}$-regularized optimal control problem. The reason is that the Gr\"ochenig-Schur class of matrices enjoys the following fundamental inclusion property:
\[ \mathcal{S}_{q_1,w}(\GG) \subset \mathcal{S}_{q_2,w}(\GG) \]
for all $0 < q_2 < q_1 \leq \infty$. This implies that the space of spatially decaying systems over $\mathcal{S}_{q,w}(\GG)$ for  $0 < q < 1$ is larger than the space of spatially decaying system over $\mathcal{S}_{1,w}(\GG)$, which was originally studied in \cite{mjieee08, mjieee09}.

\subsection{Receding Horizon Control with State and Input Constraints}  The machinery developed in this paper can be used to
analyze the spatial structure of a broader range of optimal control problems such as constrained   finite horizon control (also known as Model Predictive Control). In \cite{mjieee09}, the problem of constrained finite horizon control for the class of discrete-time spatially decaying systems over $\mathcal{S}_{q,w}(\GG)$ was considered. The results of this paper along with the techniques developed in  \cite{mjieee09} can be applied to study sparsity and localization properties of the class of multi-parametric quadratic programming problems over $q$-Banach algebras.

\subsection{Systems Governed by Partial Differential Equations}

A natural generalization of the present results is to study spatial localization of feedback control laws for linear systems that are defined using differential operators/pseudodifferential operators/integral operators. The LQR problem for this class of systems involve state space and weight matrices $A,B, Q, R$ whose elements are
differential operators/pseudodifferential operators/integral operators. A special class of such systems  includes translation invariant operators, such as partial differential operators with constant coefficients, spatial shift operators, spatial convolution operators, or a linear combination of several of such operators  \cite{HormanderVol1},
or general pseudodifferential operator and integral operators \cite{trevesbook, grochenigbook}. In general, some of the operators may be unbounded, so the notion of a solution for the PDE system requires some deep mathematics, such as theory of operator semigroups, theory of distributions, and theory of Fourier integral operators \cite{curtain-book, gelfands}. For translation invariant systems, the feedback control law has the spatial convolution form. The main idea is to exploit spatial invariance property of such systems by taking a spatial Fourier transform and study the diagonalized form (in the Fourier domain) of the corresponding LQR problem
\cite{bamiehPD02,curtain-sasane-2011}. The Fourier approach is only applicable to the class of spatially invariant systems. To generalize the present results, one would need to develop suitable operator algebras   in order to quantify the sparsity and spatial localization for the class of PDE system, this is beyond the scope of this paper and be discussed in our future works.

\subsection{Nonlinear Spatially Decaying Systems}
The notion of spatial decay can be extended to study spatially distributed systems with nonlinear dynamics. Let us consider an infinite-dimensional nonlinear system of the form
\begin{equation}
\dot{x}~=~f(x,u)\label{nonlinear-sys}
\end{equation}
where $x,u \in \ell^{2}(\GG)$ and $f: \ell^{2} (\GG) \times \ell^{2}(\GG) \rightarrow \ell^{2}(\GG)$. It is assumed that a suitable notion of a solution exists for this system. We say that system \eqref{nonlinear-sys} is spatially decaying if $f$ has continuous bounded gradients $\nabla_x f$ (with respect to $x$) and $\nabla_u f$ (with respect to $u$) on a $q$-Banach algebra for $0 < q \leq \infty$. This is a natural extension in order to develop a theory for analysis of nonlinear spatially distributed systems including quantification of the sparsity and spatial localization measures for nonlinear networks and approximation techniques. As part of this theory, one of the fundamental issues to address is the inverse-closedness property of nonlinear maps.
In \cite{sunaicmsubmitted}, it shown that under some conditions the inverse-closedness property holds for some class of nonlinear functionals  over some inverse-closed Banach algebras.

\section{Appendix}\label{appendix}

\subsection{Proof of Lemma \ref{weightexample.lem}}\label{weight.appendix}

The class of sub-exponentially decaying coupling weight functions is the most suitable class of coupling weight functions to study sparsity features of spatially decaying matrices. For the sub-exponential coupling weight function  $e_{\sigma, \delta}$, one can conclude from the following inequality
\[ (1+t)^\delta ~\le~ 1+ (2^\delta-1) t^\delta ~\le~ 1+t^\delta \ {\rm for \ all} \ t\in [0, 1]  \]
 that $e_{\sigma, \delta}$
is submultiplicative and that another sub-exponential weight function
$e_{\sigma', \delta}$ with $\sigma'=\sigma/\sqrt[\delta]{(2^\delta-1)}$ can be adopted as its companion weight \cite[Example A.3]{suntams07}.
Furthermore, the second inequality in Definition \ref{def-weight} holds with constants $D_{e}$ and $\theta_{e}$, because
\begin{eqnarray*}
& & \hspace{-2.3cm} \inf_{\tau\ge 0} \Bigg\{\Big(\sum_{|i-j|<\tau} \big(e_{\sigma', \delta}(i-j)\big)^{\frac{2q}{2-q}}\Big)^{1-\frac{q}{2}}
+  t \sup_{|i-j|\ge \tau} \left(\frac{e_{\sigma', \delta}(i-j)}{e_{\sigma, \delta}(i-j)}\right)^{q} \Bigg\}\\
& & \hspace{-0.7cm} \leq ~ \inf_{\tau\ge 0} \Bigg\{\exp \big((2^\delta-1)q\sigma^{-\delta} \tau^\delta \big)
\Big(\sum_{|i-j|<\tau} 1 \Big)^{1-\frac{q}{2}}~+~  t  \exp \big(-(2-2^\delta)q\sigma^{-\delta} \tau^\delta \big)\Bigg\}\\
&  &  \hspace{-0.7cm} \le~ 2^{d(1-\frac{q}{2})} \inf_{\tau\ge 0}\Bigg\{ \exp\Big(\big((2^\delta-1)q\sigma^{-\delta}+d (1-\frac{q}{2})\delta^{-1}\big) \tau^\delta\Big)+  t  \exp \big(-(2-2^\delta)q\sigma^{-\delta}\tau^\delta \big)\Bigg\}\\
& & \hspace{-0.7cm} \le 2^{d+1-\frac{dq}{2}} ~t^{1-\frac{q\delta (2-2^\delta)}{q\delta+d(1-\frac{q}{2})\sigma^{\delta}}}.
\end{eqnarray*}
Next, we consider the class of polynomial coupling weight functions. This class of coupling weight functions are submultiplicative, i.e.,
\[\pi_{\alpha, \sigma} (i,j)~\le ~\pi_{\alpha, \sigma} (i,k)~ \pi_{\alpha, \sigma}(k,j)\]
for all $i,j,k\in \mathbb{Z}^d$. It is straightforward to verify that
\[ \pi_{\alpha, \sigma} (i,j)~\le~ 2^\alpha ~\big(\pi_{\alpha, \sigma} (i,k)~+~\pi_{\alpha, \sigma}(k,j)\big)\]
for all $i,j,k\in \mathbb{Z}^d$.
Thus, the constant weight function $u_{\alpha, \sigma}:=\big[u_{\alpha, \sigma}(i,j)\big]_{i,j\in \mathbb{Z}^d}$ with
 $u_{\alpha, \sigma}(i,j)=2^\alpha$ can be used as a companion weight of the polynomial coupling weight function $\pi_{\alpha, \sigma}$. Moreover, the second inequality in Definition \ref{def-weight} holds with constants $D_{\pi}$ and $\theta_{\pi}$, because
\begin{eqnarray*}
\inf_{\tau\ge 0} \left\{ \Big (\sum_{|i-j|< \tau} \big(u_\alpha(i,j)\big)^{\frac{2q}{2-q}}\Big)^{1-\frac{q}{2}}+ t
\sup_{|i-j|\ge \tau} \left( \frac{u(i,j)}{w(i,j)} \right)^{q}\right\}
 & \le &  2^{q\alpha} ~\inf_{\tau\ge 0} \left\{ \Big(~\sum_{|i-j|< \tau} 1 \Big)^{1-\frac{q}{2}}~+~   t~ \left(1+\frac{\tau}{\sigma} \right)^{-q\alpha}\right\} \\
  &\le &
     2^{q\alpha+1} \max \Big\{ 1, (2\sigma)^{d(1-\frac{q}{2})}\Big\}~ t^{1-\frac{q\alpha}{q\alpha+ d(1-\frac{q}{2})}} .
\end{eqnarray*}

\subsection{Weak Submultiplicative Property Resulting from Definition \ref{def-weight}}\label{append-remark}

Suppose that $w$ is a weight that satisfies the inequalities in Definition \ref{def-weight}. For $0<q\le 1$, we get
\begin{eqnarray*}
 \sup_{j\in \GG} \left(\frac{u(i,j)}{w(i,j)}\right)^{q} & \leq &   \inf_{\tau\ge 0}\Bigg(
\sup_{j\in \GG \atop \rho(i,j)<\tau} \left(\frac{u(i,j)}{w(i,j)}\right)^{q}+ \sup_{j\in \GG \atop \rho(i,j)\ge \tau } \left( \frac{u(i,j)}{w(i,j)} \right)^{q}\Bigg)\\
&  \leq &
\inf_{\tau\ge 0} \Bigg(\Big(\sum_{j\in \GG \atop \rho(i,j)<\tau} |u(i,j)|^{\frac{2q}{2-q}} \Big)^{1-\frac{q}{2}} +
\sup_{j\in \GG \atop \rho(i,j)\ge \tau}\left(\frac{u(i,j)}{w(i,j)}\right)^{q}\Bigg) \nonumber\\
&  \le & D \quad {\rm for \ all} \ i\in \GG.
\end{eqnarray*}
The second inequality holds as $w(i,j)\ge 1$ for $i,j\in \GG$ and the last inequality follows by applying the second inequality in Definition \ref{def-weight} with $t=1$. The above estimate for the weight $w$ and its companion weight $u$ along with property
\eqref{uw.eq} prove the following weak version of the submultiplicative property \eqref{weight.submultiplication}:
\begin{equation*} 
 w(i, j) ~\le~ C_0  w(i,k) \hspace{0.05cm} w(k, j)
 \quad {\rm  for \ all} ~~ i,j,k\in \GG.
\end{equation*}
with constant $C_{0}\le 2 \sqrt[q]{D}$.
We refer the reader to  \cite{sunca11, suntams07} for similar arguments for $1<q\le \infty$.

\subsection{Proof of Theorem \ref{qtozero.thm}}
According to the definition of ${\mathcal S}_{0,1}$-measure, there exist $i_0$ and $j_0\in \GG$ such
  that
  \begin{equation} \label{qtozero.thm.pf.eq1}
  \|A\|_{{\mathcal S}_{0, 1}(\GG)}=\max \big\{ \|a_{i_0 \cdot}\|_{\ell^0(\GG)}, \|a_{\cdot  j_0}\|_{\ell^0(\GG)}\big\}.
  \end{equation}
Then, from  \eqref{sequence.asym} and \eqref{qtozero.thm.pf.eq1} it follows that
\begin{eqnarray} \label{qtozero.thm.pf.eq2}
  \lim_{q\to 0}  \max \Big\{  \sum_{j\in \GG} |a_{i_0 j}|^q w(i_0,j)^q,
\sum_{i\in \GG} |a_{i j_0}|^q w(i,j_0)^q\Big\} =   \max \big\{ \|a_{i_0 \cdot}\|_{\ell^0(\GG)}, \|a_{\cdot j_0}\|_{\ell^0(\GG)}\big\} ~= ~\|A\|_{{\mathcal S}_{0, 1}(\GG)}.
\end{eqnarray}
One observes that
\begin{eqnarray} \label{qtozero.thm.pf.eq3}
 \max \Big\{ \sum_{j\in \GG} |a_{i_0 j}|^q w(i_0,j)^q,
\sum_{i\in \GG} |a_{i j_0}|^q w(i,j_0)^q\Big\} \leq ~
\|A\|_{{\mathcal S}_{q, w}(\GG)}^q \le  M^q \|A\|_{{\mathcal S}_{0, 1}(\GG)},
\end{eqnarray}
where $M=\sup_{i,j\in \GG} |a_{ij}| w(i,j)<\infty$ by the assumption on bounded entries.
Letting $q\to 0$ in \eqref{qtozero.thm.pf.eq3} and applying
\eqref{qtozero.thm.pf.eq2}, we establish the limit \eqref{lqmatrixlimit}.

\subsection{Proof of Theorem \ref{schur-cor}}
\label{schur.appendix}

First, we show that ${\mathcal S}_{q, w}(\GG)$ is a $q$-Banach algebra. It is straightforward to verify properties (i) and (ii) in Definition \ref{def-q-banach0}. Therefore, we only need to prove properties (iii) and (iv)  in Definitions \ref{def-q-banach0} and \ref{def-q-banach}. For every pair of arbitrary matrices $A=[a_{ij}]_{i,j\in \GG}$ and $B=[b_{ij}]_{i,j\in \GG}$ in ${\mathcal S}_{q, w}(\GG)$, we write their summation by $D=A+B$ where $D=[d_{ij}]_{i,j\in \GG}$. Then,
\begin{eqnarray}
 \sup_{i\in \GG} \sum_{j\in \GG} |d_{ij}|^q w(i,j)^{q}
  & \le &  \sup_{i\in \GG} \sum_{j\in \GG}~ \Big(|a_{ij}|^p +|b_{ij}|^q\Big) w(i,j)^{q}\nonumber \\
 & \leq & \|A\|_{{\mathcal S}_{q, w}}^q + \|B\|_{{\mathcal S}_{q, w}}^q. \label{triangle-1}
\end{eqnarray}
The first inequality holds due to the well-known inequality $(\alpha+\beta)^q\le \alpha^q+\beta^q$ for all nonnegative numbers $\alpha$ and $\beta$.
In a similar manner, one can easily verify that
\begin{equation}
  \sup_{j\in \GG} ~\sum_{i\in \GG} ~|d_{ij}|^q ~w(i,j)^{q}
  ~\le~  \|A\|_{{\mathcal S}_{q, w}}^q~+~ \|B\|_{{\mathcal S}_{q, w}}^q. \label{triangle-2}
\end{equation}
Inequalities (\ref{triangle-1}) and (\ref{triangle-2}) imply that the $q$-norm is $q$-subadditive and that $\mathcal{S}_{q,w}$--measure satisfies property (iii).

To prove property (iv), if the product of $A$ and $B$ is denoted by $C=AB$ where $C=[c_{ij}]_{i,j\in \GG}$, then we have
\begin{eqnarray}\label{product.pf.eq1}
 \sup_{i\in \GG} ~\sum_{j\in \GG} |c_{ij}|^q w(i,j)^{q}
 & \le & \sup_{i\in \GG} \sum_{j, k\in \GG} |a_{ik}|^q ~w(i,k)^{q} ~|b_{kj}|^q w(k,j)^{q}\nonumber\\
& \leq & ~\|A\|_{{\mathcal S}_{q, w}}^q \|B\|_{{\mathcal S}_{q, w}}^q.
\end{eqnarray}
The first inequality uses the submultiplicative property \eqref{weight.submultiplication} for the coupling weight function $w$ and the following monotonic inequality
\begin{equation}
\label{ellqnormmonotone}
\|z\|_1 \le \|z\|_q
\end{equation}
for all $z \in \ell^q(\GG)$ and $0< q \leq 1$. Through similar proof steps, one can verify that
\begin{equation} \label{product.pf.eq2}
\sup_{j\in \GG}~ \sum_{i\in \GG} ~\big |c_{ij}\big|^q ~w(i,j)^{q}
~\le~   ~\|A\|_{{\mathcal S}_{q, w}}^q ~\|B\|_{{\mathcal S}_{q, w}}^q.
\end{equation}
From inequalities \eqref{product.pf.eq1} and \eqref{product.pf.eq2}, we conclude that the $q$-norm is submultiplicative.

In the next step, we prove that ${\mathcal S}_{q, w}(\GG)$ is a proper $q$-Banach algebra. It is straightforward to verify property ({\bf P1}). We prove a more general form of property ({\bf P2}). For every matrix
$A=[a_{ij}]_{i,j\in \GG}$ in ${\mathcal S}_{q, w}(\GG)$, the following inequality holds for all $1 \leq p \leq \infty$,
\begin{equation*}
\|A\|_{{\mathcal B}(\ell^p(\GG))}\le \max\Big \{\sup_{i\in \GG}\sum_{j\in \GG} |a_{ij}|,\  \sup_{j\in \GG}\sum_{i\in \GG} |a_{ij}|\Big\}.
\end{equation*}
More details about this inequality can be found in \cite[Theorem 2.4]{suntams07}. On the other hand, we have the following inequality
\begin{eqnarray*}
\sup_{i\in \GG}~\sum_{j\in \GG} |a_{ij}|
& \le &
\sup_{i\in \GG}\Big(\sum_{j\in \GG}~ |a_{ij}|^q\Big)^{1/q}\\
& \le &
\sup_{i\in \GG}\Big(\sum_{j\in \GG}~ |a_{ij}|^q ~w(i,j)^{q}\Big)^{1/q}  ~\le~  \|A\|_{{\mathcal S}_{q, w}},
\end{eqnarray*}
where the first inequality follows from
\eqref{ellqnormmonotone} and the second inequality holds because of property (i) of a  coupling weight function. Similarly, one can verify that
\begin{eqnarray*}
\sup_{j\in \GG}\sum_{i\in \GG} ~|a_{ij}|
& \le &
\sup_{i\in \GG}\Big(\sum_{j\in \GG} |a_{ij}|^q ~w(i,j)^{q}\Big)^{1/q} ~\le~  \|A\|_{{\mathcal S}_{q, w}}.
\end{eqnarray*}
From the above three inequalities, we can conclude that
\[ \|A\|_{\mathcal{B}(\ell^{p}(\GG))}~\leq~\|A \|_{\mathcal{S}_{q,w}(\GG)} \]
for all $A \in \mathcal{S}_{q,w}(\GG)$. This implies that ${\mathcal S}_{q, w}(\GG)$ is  continuously embedded in ${\mathcal B}(\ell^p(\GG))$ for all $1 \leq p \leq \infty$. Therefore, inequality \eqref{wiener.thm.eq4} holds when $p=2$. We should mention that the main idea of this part of our proof stems from the arguments presented in references \cite{gltams06} and \cite{suntams07}.

 Finally, we verify the differential norm property for ${\mathcal S}_{q, w}(\GG)$. For $q=1$, ${\mathcal S}_{q, w}(\GG)$ has the differential norm property, we refer to   \cite{suntams07} for a proof. Therefore from now on we suppose that $0 < q < 1$. For every $A=[a_{ij}]_{i,j\in \GG}$ and $B=[b_{ij}]_{i,j\in \GG}$ in ${\mathcal S}_{q, w}(\GG)$, let us denote
$C=AB$ where $C=[c_{ij}]_{i,j\in \GG}$. From property \eqref{uw.eq}, it follows that
\begin{eqnarray}
 \sup_{i\in \GG} \sum_{j\in \GG}|c_{ij}|^q w(i,j)^{q}  & \le &  \sup_{i\in \GG} \sum_{j, k\in \GG}  |a_{ik}|^q |b_{kj}|^q w(i,k)^{q} u(k,j)^{q}~+~ \sup_{i\in \GG} \sum_{j, k\in \GG} |a_{ik}|^q |b_{kj}|^q u(i,k)^{q} w(k,j)^{q} \nonumber\\
 &  \le &  \|A\|_{{\mathcal S}_{q,w}}^q \sup_{k\in \GG} \sum_{j\in \GG} |b_{kj}|^q u(k,j)^{q}  +  \|B\|_{{\mathcal S}_{q, w}}^q \sup_{i\in \GG} \sum_{k\in \GG} |a_{ik}|^q u(i,k)^{q}. \nonumber
\end{eqnarray}
In the following, we obtain upper bounds for each term in the right hand side of the last above inequality. Let us first consider
\begin{eqnarray}
 \sup_{k\in \GG} \sum_{j\in \GG} |b_{kj}|^q u(k,j)^{q} & = & 
 \sup_{k\in \GG}~
 \inf_{\tau \ge 0}
\Bigg( \sum_{j\in \GG \atop \rho(k,j)<\tau}+\sum_{j\in \GG \atop \rho(k,j)\ge \tau} \Bigg)~ |b_{kj}|^q ~u(k,j)^{q} \nonumber\\
& & \hspace{-1.7cm} ~~\leq  \sup_{k\in \GG} ~\inf_{\tau\ge 0}
 \Bigg\{ \Bigg(\hspace{-.1cm}\sum_{j\in \GG \atop \rho(k,j)<\tau} \hspace{-0.2cm} |b_{kj}|^2\Bigg)^{\frac{q}{2}}   \Bigg(\sum_{j\in \GG \atop \rho(k,j)< \tau} \hspace{-0.2cm} |u(k,j)|^{\frac{2q}{2-q}}\Bigg)^{1-\frac{q}{2}}+  \|B\|_{{\mathcal S}_{q, w}}^q \sup_{j\in \GG \atop \rho(k,j)\ge \tau} \left(\frac{ u(k, j)}{w(k,j)}\right)^{q}\Bigg\} \nonumber\\
& & \hspace{-1cm} ~~ \leq \sup_{k\in \GG} ~\inf_{\tau\ge 0}
 \Bigg\{ ~ \|B\|_{{\mathcal B}(\ell^2)}^q~
 \Bigg(\sum_{j\in \GG \atop \rho(k,j)< \tau} |u(k,j)|^{\frac{2q}{2-q}}\Bigg)^{1-\frac{q}{2}} +  \|B\|_{{\mathcal S}_{q, w}}^q ~\sup_{j\in \GG \atop \rho(k,j)\ge \tau} ~\left(\frac{ u(k, j)}{w(k,j)}\right)^{q}\Bigg\}\nonumber\\
& \le &  D ~ \|B\|_{{\mathcal B}(\ell^2)}^{q\theta}~ \|B\|_{{\mathcal S}_{q,w}}^{q(1-\theta)}. \label{basicproperty.pr.iii.eq2}
\end{eqnarray}
In the last two inequalities, we apply the H\"older's inequality, inequalities in Definition \ref{def-weight} with $t=\frac{\|B\|_{{\mathcal S}_{q,w}}^{q}}{\|B\|_{{\mathcal B}(\ell^2)}^{q}} \geq 1$, and the following inequality
\[
\sum_{j\in \GG \atop \rho(k,j)<\tau} |b_{kj}|^2 ~\le~
\sum_{j\in \GG}~ |b_{kj}|^2 =\|B^* e_k\|_2^2\nonumber \\
 \le  \|B^*\|_{{\mathcal B}(\ell^2)}^{2}= \|B\|_{{\mathcal B}(\ell^2)}^{2},
\]
where $e_k=[\delta_{kj}]_{j\in \GG}$ is the $k$'th standard Euclidean basis.  In a similar manner, we can show that
\begin{eqnarray}
& & \hspace{-0.75cm} \sup_{i\in \GG} \sum_{k\in \GG} ~|a_{ik}|^q ~u(i,k) ~\leq ~\nonumber \\
& & \hspace{-0.75cm}
\sup_{i\in \GG}~\inf_{\tau\ge 0}~\Bigg\{~
 \|A\|_{{\mathcal B}(\ell^2)}^q \Big(\sum_{k\in \GG  \atop \rho(i,k)< \tau} (u(i,k))^{\frac{2q}{2-q}}\Big)^{1-\frac{q}{2}} ~+\nonumber\\
 & & \hspace{3.2cm}
 \|A\|_{{\mathcal S}_{q, w}}^q \sup_{k\in \GG \atop \rho(i,k)>\tau} ~\left(\frac{ u(i, k)}{w(i,k)}\right)^{q}\Bigg\}\nonumber
\end{eqnarray}
\begin{eqnarray}
& & \hspace{-4.5cm}
\leq D ~ \|A\|_{{\mathcal B}(\ell^2)}^{q\theta} ~\|A\|_{{\mathcal S}_{q,w}}^{q(1-\theta)}.\label{basicproperty.pr.iii.eq3}
\end{eqnarray}
By combining inequalities  \eqref{basicproperty.pr.iii.eq2} and \eqref{basicproperty.pr.iii.eq3}, we get
\begin{eqnarray}
& & \hspace{-1.9cm} \sup_{i\in \GG} ~\sum_{j\in \GG} ~|c_{ij}|^q ~w(i,j)^{q}~\le~ D ~\|A\|_{{\mathcal S}_{q, w}}^q \|B\|_{{\mathcal S}_{q, w}}^{q}~ \times \nonumber \\
& & 
 \left( \left(\frac{\|A\|_{{\mathcal B}(\ell^2)}}{\|A\|_{{\mathcal S}_{q, w}}}\right)^{q\theta}
 +  \left(\frac{\|B\|_{{\mathcal B}(\ell^2)}}{\|B\|_{{\mathcal S}_{q, w}}}\right)^{q\theta}\right). \nonumber
\end{eqnarray}
Using a similar argument along with property (ii) as in Definition \ref{def-weight} for a weight function, i.e., symmetry, we get
\begin{eqnarray}
& & \hspace{-1.9cm} \sup_{j\in \GG}~ \sum_{i\in \GG} ~|c_{ij}|^q ~w(i,j)^{q} ~\le~ D ~\|A\|_{{\mathcal S}_{q, w}}^q \|B\|_{{\mathcal S}_{q, w}}^{q}~\times \nonumber \\
& & 
\left( \left(\frac{\|A\|_{{\mathcal B}(\ell^2)}}{\|A\|_{{\mathcal S}_{q, w}}}\right)^{q\theta}
 +  \left(\frac{\|B\|_{{\mathcal B}(\ell^2)}}{\|B\|_{{\mathcal S}_{q, w}}}\right)^{q\theta}\right).\nonumber
 \end{eqnarray}
Therefore, we can conclude that
\[
\|AB\|_{{\mathcal S}_{q, w}}^q \le  D \|A\|_{{\mathcal S}_{q, w}}^q \|B\|_{{\mathcal S}_{q, w}}^{q}
\left( \left(\frac{\|A\|_{{\mathcal B}(\ell^2)}}{\|A\|_{{\mathcal S}_{q, w}}}\right)^{q\theta}
 +  \left(\frac{\|B\|_{{\mathcal B}(\ell^2)}}{\|B\|_{{\mathcal S}_{q, w}}}\right)^{q\theta}\right).
 \]

\subsection{ Proof of Theorem \ref{wiener.thm}}\label{wienerthm.proofappendix}

For every $ n\ge 1$,  let us write $n=\sum_{j=0}^{N} \epsilon_j 2^j$ with
$ \epsilon_N=1$ and $\epsilon_j \in \big \{0, 1 \big \}$ for $0\le j\le N-1$. For a given $A \in \mathcal{A}$, by performing induction on the differential norm property we obtain
\begin{eqnarray}
 \hspace{-0.05cm} \|A^n \|_{\mathcal A}^q  & \le  &
K_0^{q\epsilon_0} \|A\|_{\mathcal A}^{q\epsilon_0} \|A^{n-\epsilon_0}\|_{\mathcal A}^q\nonumber\\
& \le & 2D K_0^{q\epsilon_0}  \|A\|_{\mathcal A}^{q\epsilon_0}
\|A\|_{{\mathcal B}(\ell^2)}^{(n-\epsilon_0)q \theta/2}\|A^{(n-\epsilon_0)/2}\|_{\mathcal A}^{q(2-\theta)}\nonumber\\
& \le & \cdots\\
& \le &
\big(2D\big)^{\sum_{j=0}^{N-1}  (2-\theta)^j} ~K_{0}^{q\sum_{j=0}^{N-1} \epsilon_j (2-\theta)^j} ~\times\nonumber\\
& & \hspace{1.7cm} \left(\frac{\|A\|_{\mathcal A}}{\|A\|_{{\mathcal B}(\ell^2)}}\right)^{q\sum_{j=0}^{N} \epsilon_j (2-\theta)^j}  \|A\|_{{\mathcal B}(\ell^2)}^{qn}\nonumber\\
& & \hspace{-2cm} \le~
\left\{
\begin{array}{l}
\big(2DK_{0}^q\big)^{(1-\theta)^{-1} n^{\log_2 (2-\theta)}}
\|A\|_{{\mathcal B}(\ell^2)}^{qn} ~\times \\
 \hspace{0cm} \left(\frac{\|A\|_{\mathcal A}}{\|A\|_{{\mathcal B}(\ell^2)}}\right)^{q \frac{2-\theta}{1-\theta} n^{\log_2 (2-\theta)}}   \hspace{1cm}   {\rm if} \quad\ 0<\theta<1,\\
 \\
\big(2DK_{0}^q\big)^{\log_2 n}
 \left(\frac{\|A\|_{\mathcal A}}{\|A\|_{{\mathcal B}(\ell^2)}}\right)^{q \log_2 n+q}
\|A\|_{{\mathcal B}(\ell^2)}^{qn} \quad {\rm if} \quad \theta=1,\\
\end{array}\right.
\label{nthproduct.equation}
\end{eqnarray}
where $D, K_{0}, \theta$ are the positive constants in Definitions \ref{def-q-banach}--\ref{def-proper} (cf. \cite{suncasp05} and \cite{sunaicmsubmitted}). We will use this result in the following proof steps. For a given matrix $A \in {\mathcal A}$ with $A^{-1}\in {\mathcal B}(\ell^2(\GG))$, let us define  \[B= I- \|A\|_{{\mathcal B}(\ell^2)}^{-2}~ A^* A.\]
It is straightforward to verify that
\begin{equation}\label{wiener.thm.pf.eq3}
0 \hspace{0.05cm} \preceq \hspace{0.05cm} B \hspace{0.05cm}\preceq \hspace{0.05cm} r_0 I,
\end{equation}
where
$ r_0=1- \big(\hspace{0.05cm} \|A^{-1}\|_{{\mathcal B}(\ell^2)}~ \|A\|_{{\mathcal B}(\ell^2)}\hspace{0.05cm}\big)^{-2}
\in \big[0,\hspace{0.05cm} 1 \big)$.
When $r_0=0$, the conclusion follows immediately  as in this case we have
$A^{-1}=  \|A^{-1}\|_{{\mathcal B}(\ell^2)}^{-2}A^*$. Thus from now on, we assume that $r_0\in (0, 1)$.
We apply properties of $\|\cdot\|_{\mathcal A}$ in Definitions \ref{def-q-banach0},
 \ref{def-q-banach} and \ref{def-proper}, and we get
\begin{eqnarray}
\|B\|_{{\mathcal A}}^q & \le & \|I\|_{{\mathcal A}}^q ~+~
\|A\|_{{\mathcal B}(\ell^2)}^{-2q}~\| A^* A\|_{{\mathcal A}}^q \nonumber \\
& \le &
M^q ~+~ K_{0}^q~\|A\|_{{\mathcal B}(\ell^2)}^{-2q}~\|A\|_{{\mathcal A}}^{2q}. \label{wienerlemma.thm.pf.eq6}
\end{eqnarray}
For $\theta\in (0, 1)$, this together with \eqref{wiener.thm.pf.eq3} and inequality \eqref{nthproduct.equation} leads to
\begin{eqnarray*}
\|B^n\|_{{\mathcal A}}^q
 & \le  & \big(2DK_{0}^q\big)^{(1-\theta)^{-1} n^{\log_2 (2-\theta)}} ~\|B\|_{{\mathcal B}(\ell^2)}^{qn}~
 \left(\frac{\|B\|_{\mathcal A}}{\|B\|_{{\mathcal B}(\ell^2)}}\right)^{q\frac{2-\theta}{1-\theta} n^{\log_2 (2-\theta)}}
  \\
  & \le &   \big(2DK_{0}^q\big)^{(1-\theta)^{-1} n^{\log_2 (2-\theta)}} ~r_0^{qn}~ \left(~\frac{M^q+K_{0}^q \|A\|_{{\mathcal B}(\ell^2)}^{-2q}~\|A\|_{{\mathcal A}}^{2q}}{r_0^q}~\right)^{\frac{2-\theta}{1-\theta} n^{\log_2(2-\theta)}}
 \end{eqnarray*}
for all $n\ge 1$.
Similarly for $\theta=1$,
\begin{eqnarray*}
\|B^n\|_{{\mathcal A}}^q
  & \le &   \big(2DK_{0}^q\big)^{\log_2 n} ~r_0^{qn}~ \left(~\frac{M^q+K_{0}^q \|A\|_{{\mathcal B}(\ell^2)}^{-2q}~\|A\|_{{\mathcal A}}^{2q}}{r_0^q}~\right)^{\log_2 n+1}
 \end{eqnarray*}
 for $n\ge 1$.
Therefore  $I-B$ is invertible in ${\mathcal A}$,
and  
\begin{eqnarray*}
\hspace{-0cm}\big\|(I-B)^{-1}\big\|_{\mathcal A}^q & \le &  \sum_{n=1}^\infty ~\|B^n\|_{\mathcal A}^q 
<~  \infty\end{eqnarray*}
as $r_0\in (0, 1)$ and $\theta\in (0, 1]$. 
We recall the matrix identity
\[A^{-1}~=~ (A^*A)^{-1} A^*~=~ \|A\|_{{\mathcal B}(\ell^2)}^{-2} ~(I-B)^{-1} A^*.\]
From this, we can conclude that $A$ is invertible in ${\mathcal A}$.

\subsection{Proof of Theorem \ref{thm-spectral}}
The inclusion $\sigma_{\mathcal{A}}(A)\subset \sigma_{{\mathcal B}(\ell^2(\GG))}(A)$
follows directly from the result of Theorem \ref{wiener.thm},
while the reverse inclusion $\sigma_{{\mathcal B}(\ell^2(\GG))}(A)\subset \sigma_{\mathcal{A}}(A)$ holds by
continuous imbedding of ${\mathcal A}$ in ${\mathcal B}(\ell^2(\GG))$.

\subsection{$q$-Banach Algebras of Block Matrices}\label{appendix-block-matrix}
For a given proper $q$-Banach subalgebra $\mathcal{A}$ of ${\mathcal B}(\ell^2(\GG))$, let us define an algebra of $n \times n$  matrices with entries in ${\mathcal A}$ by
\begin{equation}\label{m2a.def}
 \mathcal{M}_{n}(\mathcal{A})=\left\{{\bf A}:=\left[a_{ij}\right]\Big| \|{\bf A}\|_{\mathcal{M}_n({\mathcal A})}<\infty\right\},\end{equation}
 where
\begin{equation}\label{m2anorm.def}  \|{\bf A}\|_{\mathcal{M}_n({\mathcal A})}=\Bigg(~\sum_{i,j=1}^{n} \|a_{ij}\|_{\mathcal A}^q~\Bigg)^{1/q}.\end{equation}

\begin{lem}\label{propermatrix.lem}
For  $ 0 < q \leq 1$, let us assume that ${\mathcal A}$ is a proper $q$-Banach subalgebra of ${\mathcal B}(\ell^2(\GG))$. Then, the algebra ${\mathcal M}_n({\mathcal A})$ is a proper $q$-Banach subalgebra of ${\mathcal B}(\ell^2(\GG\times \ldots \times \GG))$.
\end{lem}

\begin{proof} By \eqref{m2a.def}  and \eqref{m2anorm.def},
 one may verify that
the quasi-norm $\|\cdot\|_{\mathcal{M}_n({\mathcal A})}$ satisfies the first three requirements in Definition
\ref{def-q-banach}  for a $q$-Banach algebra.
Given ${\bf A}:=[a_{ij}]$ and ${\bf B}:=[b_{ij}]$
in $\mathcal{M}_n({\mathcal A})$, one may verify that
\begin{eqnarray} \label{riccati.thm.pf.eq17}
\|{\bf A}{\bf B}\|_{\mathcal{M}_n({\mathcal A})}^q &\le &
\sum_{i,j=1}^n \Big\|\sum_{k=1}^n a_{ik}b_{kj}\Big\|_{\mathcal A}^q\nonumber\\
& \le &  K_{0} ~\|{\bf A}\|_{\mathcal{M}_n({\mathcal A})}^q ~\|{\bf B}\|_{\mathcal{M}_n({\mathcal A})}^q,
\end{eqnarray}
 where $K_{0}$ is the constant in the definition of the $q$-Banach algebra ${\mathcal A}$.
 Therefore $\mathcal{M}_n({\mathcal A})$ equipped with $q$-norm $\|\cdot\|_{\mathcal{M}_n({\mathcal A})}$
 is a $q$-Banach algebra. The $q$-Banach algebra $\mathcal{M}_n({\mathcal A})$ is closed under the complex conjugate operation by
\eqref{wiener.thm.eq2} and \eqref{m2anorm.def}. For every matrix
$${\bf A}:=[a_{ij}] \in \mathcal{M}_n({\mathcal A})$$
and $x=[x_{1},\ldots,x_{n}]$ with $x_{1},\ldots, x_{n}\in \ell^2(\GG)$ and
\[\|x_{1}\|_{\ell^2(\GG)}^2+\ldots+
\|x_{n}\|_{\ell^2(\GG)}^2\le 1,\] we obtain that
\begin{eqnarray*}
 \big\|{\bf A}x  \big\|_{\ell^2(\GG\times \ldots \times \GG)}^2 & = & \sum_{i=1}^n \Big\|\sum_{j=1}^{n}a_{ij}x_{j}\Big\|_{\ell^2(\GG)}^2\\
& \le & \sum_{i=1}^n \bigg(\sum_{j=1}^{n}\|a_{ij}\|_{{\mathcal B}(\ell^2(\GG))} \|x_{j}\|_{\ell^2(\GG)}\bigg)^2\\
& \le & \sum_{i=1}^n \Big( \sum_{j=1}^{n}\|a_{ij}\|_{{\mathcal B}(\ell^2(\GG))}^2 \Big) \le  \|{\bf A}\|_{\mathcal{M}_n({\mathcal A})}^2,
 \end{eqnarray*}
 where the last inequality follows from \eqref{wiener.thm.eq4} and \eqref{m2anorm.def}.
Thus the $q$-Banach algebra $\mathcal{M}_n({\mathcal A})$ is a subalgebra of ${\mathcal B}(\ell^2(\GG\times \ldots \times \GG))$ and continuously embedded w.r.t. it, and
\begin{equation*}
\|{\bf A}\|_{{\mathcal B}(\ell^2(\GG\times \ldots \times\GG))}\le \|{\bf A}\|_{\mathcal{M}_n({\mathcal A})}
\end{equation*}
for all ${\bf A}\in \mathcal{M}_n({\mathcal A})$. The $q$-Banach algebra $\mathcal{M}_n({\mathcal A})$ is a differential subalgebra of ${\mathcal B}(\ell^2(\GG\times \ldots \times \GG))$ because for any ${\bf A}:=[a_{ij}]$ and ${\bf B}:=[b_{ij}]$ in $\mathcal{M}_n({\mathcal A})$, we get that
\begin{eqnarray*}
 \|{\bf A}{\bf B}\|_{\mathcal{M}_n({\mathcal A})}^q & \le &
\sum_{i,j=1}^n \Big\|\sum_{k=1}^n a_{ik}b_{kj}\Big\|_{\mathcal A}^q\\
&\le &  D
\sum_{i,j, k=1}^n \|a_{ik}\|_{\mathcal A}^q \|b_{kj}\|_{\mathcal A}^q \Bigg(
\bigg(\frac{\|a_{ik}\|_{{\mathcal B}(\ell^2(\GG))}} {\|a_{ik}\|_{\mathcal A}}\bigg)^{q\theta}+
\bigg(\frac{\|b_{kj}\|_{{\mathcal B}(\ell^2(\GG))}} {\|b_{kj}\|_{\mathcal A}}\bigg)^{q\theta}\Bigg)\\
& \le & n D~
\|{\bf A}\|_{\mathcal{M}_n({\mathcal A})}^q \|{\bf B}\|_{\mathcal{M}_n({\mathcal A})}^q  \Bigg(
\bigg(\frac{\|{\bf A}\|_{{\mathcal B}(\ell^2(\GG\times \ldots \times \GG))}} {\|{\bf A}\|_{\mathcal{M}_n({\mathcal A})}}\bigg)^{q\theta}+
\bigg(\frac{\|{\bf B}\|_{{\mathcal B}(\ell^2(\GG\times \ldots \times \GG))}} {\|{\bf B}\|_{\mathcal{M}_n({\mathcal A})}}\bigg)^{q\theta}\Bigg).
\end{eqnarray*}
From this, one can conclude that $\mathcal{M}_n({\mathcal A})$ is a proper $q$-Banach subalgebra of  ${\mathcal B}(\ell^2(\GG\times \ldots \times \GG))$.
\end{proof}

\end{document}